\documentclass[10pt]{article}
\usepackage{amsmath}
\usepackage{amsfonts}
\usepackage{amsthm}
\usepackage[utf8]{inputenc}
\usepackage[T1]{fontenc}
\usepackage[english]{babel}
\usepackage{amssymb}
\usepackage{mathbbol}
\usepackage[a4paper,left=2cm,right=2cm,bindingoffset=5mm]{geometry}

\theoremstyle{plain}
\newtheorem{defi}{Definition}[section]
\newtheorem{teo}{Theorem}[section]
\newtheorem{prop}{Proposition}[section]
\newtheorem{cor}{Corollary}[section]
\newtheorem{lemma}{Lemma}[section]
\newtheorem{ex}{Example}[section]

\theoremstyle{definition}

\newtheorem{oss}{Remark}[section]

\newcommand\numberthis{\addtocounter{equation}{1}\tag{\theequation}}

\def    \bs     {\boldsymbol}
\title{\textbf{Universality of free homogeneous sums \\in every dimension}}
\date{}
\author{R. Simone \\ \tiny{Università degli Studi della Basilicata} \\ \tiny{Via dell'Ateneo Lucano, 10}\\ \tiny{85100 Potenza - Italy} \footnote{The paper was realized while the author was visiting the University of Luxembourg, joining the research team of Professor Giovanni Peccati.}} 

\begin{document}
\maketitle
\linespread{1.2}

\abstract{

We prove a general multidimensional invariance principle for a family of U-statistics based on freely independent non-commutative random variables of the type $U_n(S)$, where $U_n(x)$ is the $n$-th Chebyshev polynomial and $S$ is a standard semicircular element on a fixed $W^{\ast}$-probability space.
As a consequence, we deduce that homogeneous sums based on random variables of this type are universal with respect to both semicircular and free Poisson approximations.

Our results are stated in a general multidimensional setting and can be seen as a genuine extension of some recent findings by Deya and Nourdin; our techniques are based on the combination of the free Lindeberg method and the Fourth moment Theorem.}

\textbf{Keywords}:Chebyshev polynomials; Free Probability; Homogeneous sums; Lindeberg method; Universality; Wigner chaos.

\textbf{AMS subject classifications}: 46L54; 60H05; 60F17

\tableofcontents
\section{Introduction}

Roughly speaking, a universality result (or invariance principle) is a mathematical statement implying that the asymptotic behaviour of a given random system does not depend on the distribution of its components.

The aim of this paper is to prove new universality results involving polynomials in freely independent random variables. We shall also provide explicit comparisons with several analogous phenomena in the classical setting.

Our basic framework will be the following (see the Section 2 for some relevant definitions and background results). Let $(\mathcal{A},\varphi)$ be a non-commutative probability space, and let $\{S_i\}_i$ be a collection of freely independent standard semicircular random variables defined on it. Let $d\geq 2$ be an integer, and 
$$Q_N(x_1,\dots,x_N) = \sum\limits_{i_1,\dots,i_d=1}^{N}f_N(i_1,\dots,i_d)x_{i_1}\cdots x_{i_d}$$ be a homogeneous polynomial in non-commuting variables of degree $d$. We shall assume that the kernels $f_N:[N]^{d}\rightarrow \mathbb{R}$ are mirror symmetric functions. 

The following statement contains findings from \cite[ Theorem 1.3, 1.6]{NourdinPeccatiSpeicher} (Part A) and from \cite[Theorem 1.3]{NourdinDeya} (Part B) :
\begin{teo}
\label{premessa} $ $
\begin{itemize}
\item[A] (Fourth moment Theorem) If $\varphi\big(Q_N(S_1,\dots,S_N)^2\big)=1$, and $\mathcal{S}\sim \mathcal{S}(0,1)$ denotes a standard semicircular random variable, freely independent of $\{S_i\}$, the following statements are equivalent as $N$ goes to infinity:
\begin{enumerate}
\item $Q_N(S_1,\dots,S_N) \stackrel{\text{ law }}{\longrightarrow} \mathcal{S}(0,1)$;
\item $\varphi\big(Q_N(S_1,\dots,S_N)^4\big)  \longrightarrow \varphi(\mathcal{S}^4)= 2$;
\item if $g_N = \sum\limits_{i_1,\dots,i_d=1}^{N}f_N(i_1,\dots,i_d)e_{i_1}\otimes\cdots \otimes e_{i_d}$, with $\{e_j\}_j$ an orthonormal basis of $\mathrm{L}^2(\mathbb{R}_+)$, all the non trivial contractions of the kernels $g_N$ vanish in the limit, namely \linebreak $\|g_N \stackrel{r}{\smallfrown}g_N \|_{L^{2}(\mathbb{R}_{+}^{2m-2r})} \longrightarrow 0 $ for every $r=1,\dots, d-1$;
\end{enumerate}
\item[B](\textrm{Universality of semicircular elements}) If the kernels $f_N$ are symmetric functions, the following statements are equivalent as $N$ goes to infinity:
\begin{enumerate}
\item $Q_N(S_1,\dots,S_N)  \stackrel{\text{ law }}{\longrightarrow} \mathcal{S}(0,1)$;
\item $Q_N(X_1,\dots,X_N) \stackrel{\text{ law }}{\longrightarrow} \mathcal{S}(0,1)$ for any sequence $\{X_i\}_i$ of freely independent centered random variables having unit variance.
\end{enumerate}
\end{itemize}
\end{teo}

In \cite[Theorem 3.18]{Mossel}, the authors established an invariance principle for multilinear homogeneous polynomials in random variables living on a classical probability space.
The combination with the fourth moment Theorem \cite[Theorem 1]{Nualart} allowed then to prove that the Gaussian distribution satisfies a universality phenomenon for homogeneous sums with respect to the Gaussian approximation (see \cite[Theorem 1.2,1.10]{NourdinPeccati4} for both unidimensional and multidimensional frameworks). Similar results have been established for the discrete Poisson chaos (see \cite[Theorem 3.4]{PeccatiZeng}  and \cite{PeccatiZeng2}). The fourth moment Theorem was then extended to Wigner stochastic integrals, both with respect to semicircular and free Poisson approximations (see \cite{NourdinPeccatiSpeicher} and \cite[Theorem 1.4]{NourdinPeccati} respectively). See moreover \cite[Theorem 1.3]{NouSpeiPec} for a multidimensional version of the fourth moment Theorem as to semicircular approximations. 
In \cite{NourdinDeya} the authors provided an invariance principle for homogeneous polynomials in freely independent random variables living in a non-commutative probability space: as a consequence, they were able to deduce Part B of Theorem \ref{premessa}, showing that the semicircular distribution behaves universally for homogeneous sums (with symmetric kernels) with respect to semicircular distribution, providing therefore the free counterpart to \cite{NourdinPeccati4}.

\par In this paper, we are interested in the following three questions, connected to Part B of the above statement:
\begin{enumerate}
\item are there other ``universal laws''  verifying the property at Point B? In other words, is it possible to find another sequence of freely independent r.v.'s $\{Y_i\}_i$ such that if $Q_N(Y_1,\dots,Y_N)$ converges in law to a semicircular element, then $Q_N(X_1,\dots,X_N)$ has the same asymptotic behaviour for any other sequence $\{X_i\}$ of freely independent random variables?
\item Is it possible to prove a similar universality result if we consider the free Poisson distribution (or other laws) as a limit?
\item Is it possible to extend Point B of Theorem \ref{premessa} to a multidimensional setting?
\end{enumerate}
We will provide a positive answer to all the three questions in a unified way. To this aim, we will introduce the concept of \textit{Chebyshev sum}: in its simplest form (see Section 2 for the general definition), a Chebyshev sum is a polynomial of the type $$Q_N^{(h)}(x_1,\dots,x_N) = \sum_{i_1,\dots,i_d=1}^{N} f_N(i_1,\dots,i_d) U_h(x_{i_1})\cdots U_h(x_{i_d}),$$ where $U_h(x)$ denotes the $h$-th Chebyshev polynomial (of the second kind) on the interval $[-2,2]$.

Our main achievements can be summarized as follows:
\begin{itemize}
\item[-] In Section 3, we will provide an invariance principle for vectors of Chebyshev sums of any dimension, having the same nature as the main result of \cite{NourdinDeya}, which in turn generalizes the findings of \cite{Mossel} to a free probability setting;
\item[-] in Section 4, from the invariance principle and considering symmetric kernels, we will prove that vectors of Chebyshev sums based on a semicircular system are universal with respect to both semicircular and free Poisson approximations. The semicircular universality result is a genuine extension of Part B of the Theorem \ref{premessa}, showing that semicircular random variables are universal for homogeneous sums with respect to the semicircular approximation.

To our knowledge, the Poisson result is the first universality statement for the Free Poisson law proved in a free setting: in particular, for one-dimensional vectors, it is the free counterpart to \cite{NourdinPeccati4}. One should also note that, in the classical case, the only law that is known to be universal with respect to the Gamma limit is the Gaussian one, whereas our results allow one to display a new infinite collection of universal distributions with respect to the free Poisson approximations. 

More generally, our findings are the first multidimensional universality results for homogeneous sums proved in a free setting: as such, they complement \cite{NourdinPeccati4}.
\end{itemize}

To make the presentation more reader-friendly, the most technical proofs are gathered together in the last section, while the Appendix contains some relevant statements from the literature. 

\section{Preliminaries}
\subsection{Elements of free probability}
In the present subsection, we shall summarize the basic tools and results of free probability theory that will be used in the rest of the paper. Note that we only aim at giving a brief overview of the subject: the reader is referred to the fundamental references \cite{Speicher} and \cite{Voiculescu} for a more detailed presentation.
\begin{itemize}
\item[(i)]
A \textit{$W^{\ast}$-probability} space is a pair $(\mathcal{A},\varphi)$, where $\mathcal{A}$ is a von Neumann algebra of operators, with unity $1$, and $\varphi: \mathcal{A} \rightarrow \mathbb{C}$ is a unital linear functional on it, satisfying the following properties:
\begin{enumerate}
\item $\varphi$ is a trace: $\varphi(a b)= \varphi(b a)$ for every $a, b \in \mathcal{A}$;
\item $\varphi$ is positive: if $a^{\ast}$ denotes the adjoint of an element $a \in \mathcal{A}$, then $\varphi(a a^{\ast}) \geq 0$;
\item $\varphi$ is faithful: $\varphi(a a^{\ast}) = 0$ implies that $a=0$.
\end{enumerate}
\item[(ii)] In the literature, it is customary to refer to the self-adjoint elements of a $W^{\ast}$-probability space as \textit{random variables}. If $a$ is a random variable in $\mathcal{A}$, the elements of the sequence $\{\varphi(a^{m}): m \in \mathbb{N}\}$ are called the moments of $a$. A random variable $a$ with zero mean ($\varphi(a) = 0$) will be called centered; if a random variable $b$ is not centered, we call $b- \varphi(b)1$ the centering of $b$. 

For a random variable $a$, the \textit{spectral radius} is defined as $\rho(a)=\lim\limits_{k \rightarrow \infty}|\varphi(a^{2k}) |^{\frac{1}{2k}}$; if $\rho(a)$ is finite, then $a$ is called a bounded random variable. 
Indeed, for every bounded random variable $a$, there exists a real measure $\mu_{a}$ with compact support included in $[-\rho(a), \rho(a)]$ (called the law, or the \textit{distribution} of $a$), that allows us to represent the moments of $a$ (see \cite{Speicher}):
$$ m_k(a) = \varphi(a^{k}) = \int_{\mathbb{R}}x^{k}\mu_{a}(dx).$$
\item[(iii)] Thanks to the positivity of the state $\varphi$, we have the following Cauchy-Schwarz type inequality: for every $a, b \in \mathcal{A}$, $ |\varphi(a b^{\ast}) |^{2} \leq \varphi(a a^{\ast}) \varphi(b b^{\ast}).$
\item[(iv)] The unital subalgebras $\mathcal{A}_1, \dots, \mathcal{A}_n$ of $\mathcal{A}$ are said to be \textit{freely independent} if, for every $k \geq 1$, for every choice of integers $i_1,\dots, i_k$ with $i_j \neq i_{j+1}$, and random variables $a_{i_j} \in \mathcal{A}_j$, we have $\varphi(a_{i_1}a_{i_2}\cdots a_{i_k}) = 0$. Random variables $a_1,\dots, a_n$ are said to be freely independent if the (unital) subalgebras they generate are freely independent.
\item[(v)]
Recall that a partition $\pi$ of the set $[n]=\{1,2,\dots,n\}$ is a collection of nonempty and pairwise disjoint subsets of $[n]$, whose union is the whole set $[n]$. A partition $\pi$ is said to be \textit{non-crossing} if, whenever there exist integers $i < j < k < l$, with $i\sim_{\pi} k$, $j \sim_{\pi} l$, then $j \sim_{\pi} k$ (here, $i \sim_{\pi} j$ means that $i$ and $j$ belong to the same block of $\pi$). The lattice of the non-crossing partitions, denoted by $\mathcal{NC}([n])$, is the combinatorial structure underlying the free probability setting.
\item[(vi)] For $\pi \in \mathcal{NC}([n])$, the \textit{free cumulant} $r_{\pi}(a)$ of a random variable $a$ is the multiplicative function on $\mathcal{A}^n$ satisfying the formula:
$$ m_{n}(a) = \sum_{\pi \in \mathcal{NC}([n])} r_{\pi}(a),$$ 
with $r_\pi(a)= \prod_{b\in \pi} r_{|b|}(a)$, or equivalently,
$$r_n(a) = \sum_{\pi \in \mathcal{NC}([n])}  \mu(\pi, \hat{1}) m_{\pi}(a)$$
with $\mu(\pi,\hat{1})$ denoting the M\"{o}bius function on the interval $[\pi, \hat{1}]$ (see \cite[Chapter 11]{Speicher} for more details). The first four cumulants are:
\begin{enumerate}
\item  $r_1(a) = \varphi(a)$, the mean;
\item  $r_2(a) = m_2(a) - m_1(a)^2$, called the variance;
\item  $r_3(a) = 2m_1^3 + m_3 - 3m_2 m_1$;
\item $r_4(a) = m_4 -2m_2^2 +10 m_2 m_1^2  - 4m_1 m_3 - 5m_1^4$.
\end{enumerate}
\item[(vii)] A centered random variable $s \in \mathcal{A}$ is called a \textit{semicircular element} of parameter $\sigma^2 > 0$ (for short, $s \sim \mathcal{S}(0,\sigma^2)$) if its distribution is the Wigner semicircle law on $[-2\sigma,2\sigma]$ given by:
$$ \mathcal{S}(0,\sigma^2)(dx) = \dfrac{1}{2\pi \sigma^2}\sqrt{4\sigma^2 - x^{2}} dx. $$
If $\sigma=1$, $s$ is called a standard semicircular random variable. 

The even moments of a semicircular element $s$ of parameter $\sigma^2$ are given by:
$$ \int_{-2\sigma}^{2\sigma}x^{2m}\mathcal{S}(0,\sigma^2)(dx) = C_m \sigma^{2m},$$
with $\{C_m\}_{m\in \mathbb{N}}$ being the sequence of the Catalan numbers, namely $C_m = \dfrac{1}{m+1}\binom{2m}{m}$, while all its odd moments are zero. Equivalently, $r_1(s)=0$, $r_2(s) = \sigma^{2}$ and $r_n(s) = 0$ for all $n\geq 3$.
\item[(viii)]
A random variable $X(\lambda) \in \mathcal{A}$ is called a \textit{free Poisson element} of parameter $\lambda > 0$ if its distribution has the form:
$$p(\lambda)(dx) = (1-\lambda)\delta_0 + \lambda \tilde{\nu} \qquad  \text{ for } \lambda \leq 1, $$
$$ p(\lambda)(dx) = \dfrac{1}{2\pi x}\sqrt{4\lambda -(x-\lambda-1)^{2}}\; \mathbb{1}_{((1-\sqrt{\lambda})^{2},(1+\sqrt{\lambda})^{2})}(dx), \qquad \text{ for } \lambda > 1.$$
Let us denote by $Z(\lambda)$ a centered free Poisson random variable of parameter $\lambda$, namely \linebreak$Z(\lambda) = X(\lambda) - \lambda 1$. As shown in \cite[Proposition 2.4]{NourdinPeccati}, the moments of $Z(\lambda)$ are given by:
$$ \varphi\big(Z(\lambda)^{m}\big) = \sum_{j=1}^{m}\lambda^{j}R_{m,j},$$
with $R_{m,j}$ counting the number of non-crossing partitions in $\mathcal{NC}([m])$ having no singletons and having exactly $j$ blocks. In particular, if $\lambda =1$, $ \varphi\big(Z(1)^{m}\big) =  R_m$, the $m$-th Riordan number, counting the number of non-crossing partitions in $\mathcal{NC}([m])$ having no singletons. 
Equivalently, $r_1(Z(\lambda))=0$ and $r_n(Z(\lambda)) = \lambda$ for all $n\geq 2$.
\item[(ix)]
We now discuss Wigner Stochastic integration, a theory first developed in \cite{BianeSpeicher}.

For every $p: 1 \leq p < \infty$, let us denote by $L^{p}(\mathcal{A},\varphi)$ the space obtained by completion of $\mathcal{A}$ with respect to the norm $\| a\|_p = \varphi(|a|^{p})^{\frac{1}{p}}$, with $|a|$ such that $|a|^{2} = a^{\ast}a$. 

If $\{\mathcal{A}_t\}_{t \geq 0}$ denotes a filtration of  unital subalgebras of $\mathcal{A}$ (namely, $\{\mathcal{A}_t\}_{t \geq 0}$ is an increasing sequence of subalgebras: $\mathcal{A}_s \subset \mathcal{A}_t$ for $s \leq t$ ), we define a \textit{free Brownian motion} as a collection $S = \{S(t)\}_{t \geq 0}$ of self-adjoint operators in $(\mathcal{A},\varphi)$ such that:
\begin{enumerate}
\item for every $t \geq 0$, $S(t)\sim \mathcal{S}(0,t)$ and $S(t) \in \mathcal{A}_t$;
\item (stationary increments) for every $0 \leq t_1 < t_2$, the increment $S(t_2) - S(t_1)$ has the same distribution as $S(t_2-t_1)$;
\item (freely independent increments) for every $0 \leq t_1 < t_2$, the increment $S(t_2) - S(t_1)$ is freely independent of $\mathcal{A}_{t_1}$.
\end{enumerate}

Let $q \geq 2$ be an integer. A function $f \in L^{2}(\mathbb{R}_{+}^{q})$ is said to be \textit{mirror symmetric} if \linebreak $f(t_1, t_{2},\dots,t_q) = f(t_q,\dots,t_2,t_1)$ for every $t_1,\dots,t_q \in \mathbb{R}_+$.

More generally, for a complex valued kernel $f$, we say that $f$ is mirror symmetric if \linebreak $f(t_1, t_{2},\dots,t_q) = \overline{f(t_q,\dots,t_2,t_1)}$, for every $t_1,\dots,t_q \in \mathbb{R}_+$, where $\overline{f(t_q,\dots,t_2,t_1)}$ denotes the complex conjugate of $f(t_q,\dots,t_2,t_1)$. In the following, we will deal only with real-valued kernels, as in \cite{NourdinDeya}, the extension to the complex case being unnecessary for our purposes, but still approachable by the same strategy.

Given a free Brownian motion $S$ on $(\mathcal{A},\varphi)$, the construction of the Wigner stochastic integral (that is, the stochastic integral with respect to a free Brownian motion) requires exactly the same steps as those included in the definition of the classic Wiener-It\^{o} integrals with respect to a (classical) Brownian motion. 
\begin{defi}
Let $f$ be a simple function (vanishing on diagonals) in $L^{2}(\mathbb{R}_{+}^{q})$, namely \linebreak $f= \prod\limits_{j=1}^{q}\mathbf{1}_{(a_j,b_j)}$, with $(a_j,b_j)$ pairwise disjoint intervals of the real line. Then we set:
$$I_q^{S}(f) = \big(S(b_1)-S(a_1)\big)\cdots (S(b_q)-S(a_q)).$$
\end{defi}
By linearity, the last definition can be extended to every function that is a finite linear combination of simple functions vanishing on diagonals. As for the Wiener stochastic integration, for such functions the Wigner integrals satisfy the isometric relation:
$$ \langle I_q^{S}(f), I_q^{S}(g)\rangle_{L^{2}(\mathcal{A},\varphi)} = \langle f,g\rangle_{L^{2}(\mathbb{R}_+^{q})},$$
that allows us to define the Wigner integral for any $f \in L^{2}(\mathbb{R}_{+}^{q})$ (by a density argument). 
Moreover, it is easy to check that $I_q^{S}(f)$ is self-adjoint if and only if $f$ is mirror symmetric.

The sequence of the Chebyshev polynomials (of the second kind), defined by the recurrence relation $U_{0}(x) = 1$, $U_1(x) = x$, and $U_{m+1}(x) = xU_m(x) - U_{m-1}(x) \; \text{ for every } m \geq 1,$ is an orthogonal family of polynomials  with respect to the semicircle Wigner law $s(dx) = \dfrac{1}{2\pi}\sqrt{4-x^2}dx$ on the interval $[-2,2]$ (for more details, see \cite{Anshelevich, Chihara}). In the framework of the Wigner stochastic integration, this family of polynomials play the same role as the Hermite polynomials for the multiple integrals of Wiener-It\^{o} type (see e.g. \cite[Chapter 2]{Peccatilibro}).

 In particular, for every $k\geq 1$ and for every choice of integers $m_1, \dots,m_k$, it can be shown that (see \cite{Anshelevich},\cite{BianeSpeicher}):
\begin{equation}
\label{Integral}
U_{m_1}(S_{i_1})U_{m_{2}}(S_{i_2})\cdots U_{m_k}(S_{i_k}) = I_{m}^{S}\big(e_{i_1}^{\otimes m_1}\otimes e_{i_2}^{\otimes m_2}\otimes \cdots \otimes e_{i_k}^{\otimes m_k}\big),
\end{equation}
provided that $i_1 \neq i_2 \neq \cdots \neq i_k$, and $m=m_1+\cdots +m_k$, with $\{e_j\}_i$ orthonormal basis of $\mathrm{L}^{2}(\mathbb{R}^{+})$ and $\{S_j\}_j$ the associated free Brownian motion, with $S_j = I_1^{S}(e_j)$. Note that $\{S_j\}_j$ is a sequence of freely independent standard semicircular elements.
\item[(x)] Last, let us recall the connection between the free Poisson distribution with integer parameter $p$ and the standard semicircle law. Indeed, $U_2(S) \stackrel{\text{law}}{=} Z(1)$, and more generally, $Z(p) \stackrel{\text{law}}{=} \sum_{j=1}^{p}(S_j^{2}-1)$, with $S_1,\dots,S_p$ freely independent standard semicircular elements (see \cite{NourdinPeccati}).

\end{itemize}

\subsection{Notation and other preliminaries}

Let $\{x_i\}_{i \in \mathbb{N}}$ be a sequence of non-commutative variables. In the next definition we shall introduce one of the main objects of the paper.
\begin{defi}
Let $d \geq 1$ be an integer, and $\bs{h}=(h_1,\dots,h_d)$ be a vector of positive integers such that $h_i = h_{d-i+1}$ for every $i=1,\dots, \lfloor \frac{d}{2}\rfloor$ if $d\geq 2$. For every integer $N$, let $f_N:[N]^{d}\rightarrow \mathbb{R}$ be a kernel  verifying the following properties: \footnote{Of course the properties (i) and (ii) are non-trivial only if $d \geq 2$.}
\begin{itemize}
\item[(i)] mirror symmetry: $f_N(i_1,\dots,i_d) = f_N(i_d,\dots,i_1)$ for every $i_1,\dots, i_d \in \{1,\dots,N\}$;
\item[(ii)] vanishing on diagonals: $f_N(i_1,\dots,i_d) = 0$ whenever $i_j = i_k$ for $j\neq k$;
\item[(iii)] unit variance: 
\begin{equation}
\label{variance}
\mathrm{Var}(f_N) = \|f_N\|^{2}:= \sum\limits_{i_1,\dots,i_d=1}^{N}f_N(i_1,\dots,i_d)^2 = 1.
\end{equation}
\end{itemize}
Then, we define the \textbf{Chebyshev sum} of orders $\bs{h}= (h_1,\dots,h_d)$ and kernel $f_N$ by the formula:
\begin{equation}
\label{SumCheby2}
Q_N^{(\bs{h})}(f_N; x_1,\dots,x_N) = \sum_{i_1,\dots,i_d=1}^{N}f_N(i_1,\dots,i_d)U_{h_1}(x_{i_1})\cdots U_{h_{d}}(x_{i_d}).
\end{equation}
\end{defi}

Note that if we choose $h_i=1$ for every $i=1,\dots,d$, the corresponding Chebyshev sum is nothing but a homogeneous polynomial $Q_N$ of degree $d$: 
\begin{equation}
\label{QN}
Q_N(x_1,\dots, x_N) = \sum_{i_1,\dots, i_d=1}^{N}f_N(i_1,\dots,i_d)x_{i_1}\cdots x_{i_d}.
\end{equation}

\begin{oss}
The condition $h_i = h_{d-i+1}$ may look a bit artificial, but as we will see, it is needed to ensure that $Q_N^{(\bs{h})}(f_N; X_1,\dots,X_N)$ is a self-adjoint polynomial, for $X_j$ self-adjoint in $\mathcal{A}$.
\end{oss}
As in several other papers concerning our subject, many steps in the sequel will be described in terms of the contraction operators (see, for instance, \cite{NourdinDeya, NourdinPeccatiSpeicher}).

\begin{defi}
\label{starcontraction}
Let $f, g:[N]^{d} \rightarrow \mathbb{R}$. For every $r=1,\dots, d$, we define the (discrete) star contraction as the function $f \star_{r}^{r-1} g : [N]^{2d - 2r +1} \rightarrow \mathbb{R}$ given by:
\begin{align*}
f \star_{r}^{r-1} g (t_1&,\dots, t_{d-r},\gamma, s_1,\dots, s_{d-r}) = \\
&=\sum\limits_{i_1,\dots, i_{r-1}=1}^{N}f(t_1,\dots, t_{d-r},\gamma, i_1,\dots, i_{r-1})g(i_{r-1},\dots, i_{1}, \gamma,s_1,\dots,s_{d-r});
\end{align*}
and, for every $q=0,\dots, d$, we define the contraction of order $q$, $f\stackrel{q}{\smallfrown}g:[N]^{2d-2q}\rightarrow \mathbb{R}$, by the rule:
\begin{align*}
f_N \stackrel{q}{\smallfrown}& f_N (t_1,\dots, t_{d-q},s_1,\dots,s_{d-q}) = \\
&= \sum\limits_{i_{1},\dots,i_q=1}^{N}f_N(t_1,\dots,t_{d-q},i_1,\dots,i_q)f_N(i_q,\dots,i_1,s_1,\dots,s_{d-q}).
\end{align*}
\end{defi}

The contraction operator of the type $\stackrel{r}{\smallfrown}$ can be introduced for elements of the tensorial powers of any (possibly separable) Hilbert space $\mathcal{H}$, extending by linearity the following definition: for every $r=1,\dots, \min\{d,p\}$,
\begin{equation}
\big(e_{i_1}\otimes \cdots \otimes e_{i_d} \big) \stackrel{r}{\smallfrown} \big(e_{j_1}\otimes \cdots \otimes e_{j_p}\big) = \prod_{l=0}^{r-1}\langle e_{i_{d-l}}, e_{j_{l+1}}\rangle e_{i_1}\otimes \cdots \otimes e_{i_{d-r}}\otimes e_{j_{r+1}}\otimes \cdots \otimes e_{j_{p}},
\end{equation}
where $\langle \cdot, \cdot \rangle$ denotes the inner product on $\mathcal{H}$. In particular:
$$\big(e_{i_1}\otimes \cdots \otimes e_{i_d} \big) \stackrel{d}{\smallfrown} \big(e_{j_1}\otimes \cdots \otimes e_{j_d}\big) = \langle e_{i_1}\otimes \cdots \otimes e_{i_d}, e_{j_d}\otimes \cdots \otimes e_{j_1}\rangle_{\mathcal{H}^{\otimes d}} = \prod_{l=1}^{d}\langle e_{i_l}, e_{j_{d-l+1}} \rangle$$
(with $\langle \cdot,\cdot,\rangle_{\mathcal{H}^{\otimes d}}$ denoting the inner product on $\mathcal{H}^{\otimes d}$ induced by $\langle \cdot, \cdot\rangle$). Moreover observe that if $f \in \mathcal{H}^{\otimes p}$ and $g \in \mathcal{H}^{\otimes d}$, then $f \stackrel{r}{\smallfrown} g \in \mathcal{H}^{\otimes p+d-2r}$. 

\begin{ex}
If $\{e_i\}_i$ is an orthonormal sequence of $\mathcal{H}$, then:
\begin{enumerate}
\item $e_1 \otimes e_2 \otimes e_{3}  \stackrel{2}{\smallfrown} e_3 \otimes e_2 \otimes e_1 = \langle e_3,e_3\rangle \langle e_2,e_2 \rangle e_1\otimes e_{1} = e_1\otimes e_1;$
\item $e_1 \otimes e_2 \otimes e_3\stackrel{1}{\smallfrown} e_4 \otimes e_2 \otimes e_5 = \langle e_3,e_4\rangle e_1 \otimes e_2 \otimes e_2 \otimes e_5 =0$.
\item if $f_N(i,j)= \dfrac{1}{\sqrt{N-2}}$ for $i \neq j$, and $f_N(i,i) = 0$, then
$$f_N \stackrel{1}{\smallfrown} f_N(h,k) = 
\begin{cases}
1 & \text{ if } h \neq k \\
\dfrac{N-1}{N-2} & \text{ if } h=k
\end{cases}
.$$
\end{enumerate}
\end{ex}

\begin{oss}[On notation]
With a slight abuse of notation and when there is no risk of confusion, we will use the same symbol $\smallfrown$ for both the contractions of discrete kernels and the contraction operation over Hilbert spaces.  Similarly, the symbol of the norm $\|\cdot \|$ will be used for both the (square root) of the variance of a discrete kernel (as in (\ref{variance})) and for elements in the Hilbert space. Also in this case, the nature of the symbol will be clear from the context.
\end{oss}

$ $
From now on, consider fixed  a vector of positive integers $(h_1,\dots, h_d)$ with $h_i = h_{d-i+1}$ for all $i=1,\dots,\lfloor \frac{d}{2} \rfloor$, as well as a separable Hilbert space $\mathcal{H}$, with orthonormal basis $\{e_i\}_{i \in \mathbb{N}}$. For the fixed integers $h_j$'s and if $m=h_1 + \cdots + h_d$,  we will canonically associate the element $k_N$ in $\mathcal{H}^{\otimes m}$ defined by:
\begin{equation}
\label{kN1}
 k_N = \sum_{i_1,\dots,i_d=1}^{N}f_N(i_1,\dots,i_d)e_{i_1}^{\otimes h_1}\otimes \cdots \otimes e_{i_d}^{\otimes h_d},
 \end{equation}
 to every kernel $f_N:[N]^{d} \rightarrow \mathbb{R}$ as above.
\begin{oss}
In view of the constraints on $(h_1,\dots,h_d)$, $k_N$ is mirror symmetric (as a function of $m$ variables) if and only if $f_N$ is mirror symmetric (as a function of $d$ variables). 
\end{oss}

One of the staples of the entire paper is the following explicit connection between the norms of the contractions of the kernel $k_N$ defined in (\ref{kN1}) and the norms of the kernel $f_N$ (as defined in (\ref{variance})), whose proofs are straightforward.

\begin{prop}
\label{contraction5}
If $d\geq 1$, for any integer $N \geq 1$, fix positive integers $h_1,\dots, h_d\geq 1$ such that $h_i = h_{d-i+1}$ for every $i=1,\dots, \big\lfloor \frac{d}{2}\big\rfloor$ (if $d \geq 2$) \footnote{To simplify the notation, we will omit the subscripts for the norms $\|k_N \stackrel{r}{\smallfrown} k_N \|_{\mathcal{H}^{\otimes (2m-2r)}}$.}. Consider the mirror symmetric kernel given in (\ref{kN1}) in $\mathcal{H}^{\otimes m}$, where $m=h_1+\cdots +h_d$, and with mirror symmetric kernel $f_N$ over $[N]^{d}$. Then, for every $r=1,\dots, h_1+\cdots +h_d - 1$:
\begin{itemize}
\item[(i)] if $r= h_1 + \cdots + h_q$, for $q=1,\dots, d-1$, $$ \|k_N \stackrel{r}{\smallfrown} k_N \| = \|f_N \stackrel{q}{\smallfrown} f_N \|;$$
\item[(ii)] if $r= \sum\limits_{j=1}^{q-1}h_j + t$, for some $t=1,\dots, h_q -1$ and $q=1,\dots, d$, $$ \|k_N \stackrel{r}{\smallfrown} k_N \|  = \|f_N \star_{q}^{q-1} f_N \|. $$
\end{itemize}
\end{prop}

\begin{prop}
\label{contraction6}
Let $d\geq 1$. For fixed positive integers $h_1,\dots, h_d$ with $h_i = h_{d-i+1}$ for every $i=1,\dots, \big\lfloor \frac{d}{2}\big\rfloor$ (if $d \geq 2$), and such that $h_1+\cdots + h_d $ is even, consider the kernel $k_N$ given in (\ref{kN1}). Then, if $\alpha = \dfrac{h_1+\cdots +h_d}{2}$,
\begin{itemize}
\item[(i)] if $d$ is even (and so $h_1 + \cdots +h_d = 2(h_1 + \cdots + h_{\frac{d}{2}})$),
$$ \|k_N \stackrel{\alpha}{\smallfrown}k_N - k_N \| = \|f_N \stackrel{\frac{d}{2}}{\smallfrown}f_N - f_N \|;$$
\item[(ii)] if $d$ is odd, (and therefore $h_1 + \cdots +h_d = 2(h_1 + \cdots + h_{\frac{d-1}{2}}) + h_{\frac{d+1}{2}}$ is even if and only if $h_{\frac{d+1}{2}}$ is even),
$$ \|k_N \stackrel{\alpha}{\frown}k_N - k_N \| = \|f_N \star_{\frac{d+1}{2}}^{\frac{d-1}{2}}f_N - f_N \|.$$
\end{itemize}
\end{prop}

\section{Main results}
\subsection{Free Lindeberg principle}

From now on, we fix a $W^{\ast}$-probability space $(\mathcal{A},\varphi)$ (see \cite{Speicher, Voiculescu}). 

In this section, we are going to follow the approach initiated in \cite{Mossel}, leading to and state an invariance principle for vectors of Chebyshev sums in freely independent random variables.

The result we are presenting is based on the free version of the celebrated Lindeberg method and it can be seen as a generalization of the invariance principle for homogeneous sums of free random variables, proved in \cite{NourdinDeya} (see Theorem \ref{teoNourdin} in the sequel). In such a paper, the authors extended to the free setting a particular case of the invariance principle for multilinear homogeneous sums with low influences established in \cite{Mossel}.

As in the previously quoted references, of particular interest for us is the notion of ``influence'': influence functions play a prominent role in \cite{Mossel}, where the authors extend the Lindeberg method to a non-linear setting, in order to settle a number of conjectures involving the combinatorial analysis of Boolean functions. Low-influence functions were then applied in \cite{NourdinPeccati4} to obtain universality results in classical setting. See also \cite{Rotar} for some earlier related results.

Let \mbox{$f_N:[N]^{d}\rightarrow \mathbb{R}$} be mirror symmetric, vanishing on diagonals and having unit variance. For every $i=1,\dots,N$, the $i$-th \textit{influence function} associated with $f_N$ is defined as:
\begin{equation}
\label{Influence}
\mathrm{Inf}_{i}(f_N) = \sum_{l=1}^{d}\sum_{j_1,\dots,j_{d-1}=1}^{N}f_N(j_{1},\dots,j_{l-1},i,j_{l},\dots,j_{d-1})^{2}  \;.
\end{equation}
Note that, if $f_N$ is symmetric, the influence function has the expression 
$$\mathrm{Inf}_{i}(f_N) = d\sum_{j_1,\dots,j_{d-1}=1}^{N}f_N(i,j_1,\dots,j_{d-1})^{2}.$$
 Moreover, if $\| f_N\|^2 =1$, then $\sum\limits_{i=1}^{N}\mathrm{Inf}_i(f_N) = d$.
\begin{oss}
In the framework of classical probability, the definition of influence function is slightly different (see \cite[Chapter 11]{Peccatilibro}); moreover, in this case it is possible to give a nice probabilistic interpretation of $\mathrm{Inf}_i(f_N)$ as the measure of the influence that the variable $X_i$ has on the overall fluctuations of the statistic $Q_N(X_1,\dots,X_N)$, as suggested by the formula:
\begin{equation*}
(d!)^2\mathrm{Inf}_i(f_N) = \mathbb{E}[Var(Q_N(X_1,\dots,X_N))| X_j, j \neq i],
\end{equation*}
where $(X_1,\dots,X_N)$ is a vector of centered independent random variables having unit variance, and where we have used the notation:
$$ \mathbb{E}[Var(Q_N(X_1,\dots,X_N))| X_j, j \neq i] = \mathbb{E}[Q_N(X_1,\dots,X_N)- \mathbb{E}[\big(Q_N(X_1,\dots,X_N)| Y_k,k \neq i]\big)^2|Y_k, k \neq i].  $$ 
\end{oss}

The following result is the starting point of our analysis.
\begin{teo}[See \cite{NourdinDeya}]
\label{teoNourdin}
Let $(\mathcal{A},\varphi)$ be a $W^{\ast}$-probability space. Let $\mathbf{X}=\{X_i\}_{i}$ and $\mathbf{Y}= \{Y_i\}_i$ be two sequences of centered freely independent random variables, with unit variance, such that $\mathbf{X}$ and $\mathbf{Y}$ are freely independent. Suppose moreover that the $\{X_i\}_i$ (respectively $\{Y_i\}$) are either identically distributed or uniformly bounded, that is:
$$\sup_{i\geq 1}\varphi(|X_i|^{r}) < \infty \quad (\text{ resp. } \sup_{i\geq 1}\varphi(|Y_i|^{r}) < \infty).$$
If $Q_N$ denotes the homogeneous sum as in (\ref{QN}), with mirror symmetric kernel, then:
\begin{equation}
\label{InvarianceNourdin}
\varphi\big(Q_N(X_1,\dots,X_N)^{m}\big) - \varphi\big(Q_N(Y_1,\dots,Y_N)^{m}\big) = \mathcal{O}\big((\tau_N)^{\frac{1}{2}}\big)
\end{equation}
for any integer $m\geq 1$, where $\tau_N = \max\limits_{i=1,\dots,N}\mathrm{Inf}_{i}(f_N)$.
\end{teo}

Let $d\geq 2$ be an integer. For fixed integers $h_1,\dots,h_d \geq 1$, with $h_j = h_{d-j+1}$ for $j=1,\dots,\lfloor \frac{d}{2}\rfloor$, let $\bs{X}=\{X_i\}_i$ be a sequence of freely independent random variables on $(\mathcal{A},\varphi)$ such that $U_{h_j}(X_i)$ is a centered random variable with unit variance, for every $i$ and every $j=1,\dots,d$. We will consider sequences of homogeneous sums $Q_N$, whose argument is given by the vector $\bs{\mathcal{X}}^{(N)} = \big(\bs{\mathcal{X}}_1,\dots, \bs{\mathcal{X}}_N\big)$, with $\bs{\mathcal{X}}_i = (\mathcal{X}_{i,1},\dots, \mathcal{X}_{i,d})$ and $\mathcal{X}_{i,j} = U_{h_j}(X_i)$, namely:
\begin{equation}
\label{ensemble}
\bs{\mathcal{X}}_i = \big(U_{h_1}(X_i),\dots, U_{h_d}(X_i)\big).
\end{equation}

We write:
\begin{align*}
Q_N\big(\bs{\mathcal{X}}^{(N)}\big) &= \sum_{i_1,\dots,i_d=1}^{N}f_N(i_1,\dots,i_d)\mathcal{X}_{i_1,1}\mathcal{X}_{i_2,2}\cdots \mathcal{X}_{i_d,d} \\
&=  \sum_{i_1,\dots,i_d=1}^{N}f_N(i_1,\dots,i_d)U_{h_1}(X_{i_1})\cdots U_{h_d}(X_{i_d})\\
&= Q_N^{(\bs{h})}(f_N;X_1,\dots,X_N)
\end{align*}
(namely the $j$-th element in each summand is the $j$-th element in $\bs{\mathcal{X}}_{i_j}$).

\begin{oss} This further notation for Chebyshev sums facilitates the connection between our approach and the findings in \cite{Mossel}, where the authors deals with homogeneous sums in sequences of ensembles. It also simplifies the notation used in the proofs.
\end{oss}

\begin{ex}
It is obvious that a standard semicircular random variable $\mathcal{S}$ satisfies the assumptions $\varphi(U_n(\mathcal{S})^2) = 1$ and $\varphi(U_n(\mathcal{S}))=0$ for every integer $n\geq 1$. Anyway, there exist some other non trivial examples: for instance, let $d=2$ and choose $h_1 =h_2 =2$. For a bounded random variable $X$, the constraints $\varphi(U_2(X))= 0$ and $\varphi(U_2(X)^2)=1$ give $\varphi(X^2)=1$ and $\varphi(X^4) = 2$, so $X$ can be any centered random variable with unit variance and zero free fourth cumulant $r_4$ (say, a free mesokurtic variable). 
For instance, we can choose a centered free Poisson variable $Z(1)$ with parameter one, and a free symmetric Bernoulli variable freely independent of $Z(1)$, say $Y \sim \frac{1}{2}(\delta_1 + \delta_{-1})$, so that $r_2(Y)=1$ and $r_4(Y)= -1$ (see \cite{Speicher}). Therefore $X= \frac{1}{\sqrt{2}}(Z(1) + Y)$ is a centered random variable satisfying the desired hypotheses.
\end{ex}

In order to properly develop our version of the Lindeberg method (stated in the next theorem), we need to introduce some auxiliary sequences of vectors. To this aim, if $\{Y_i\}_i$ is a sequence of freely independent centered random variables with unit variance, freely independent of $\{X_i\}_i$, for every $i=1,\dots, N$ set:
\begin{equation}
\label{auxiliary}
 \bs{Z}^{N,(i)} = \bs{Z}^{(i)} = (\bs{Z}_{1}^{(i)},\dots, \bs{Z}_{N}^{(i)}) = (\bs{Y_1},\dots,\bs{Y_{i-1}}, \bs{\mathcal{X}}_{i},\dots, \bs{\mathcal{X}}_{N}),
 \end{equation}
where $\bs{Y_j}$ is the vector consisting of $d$ copies of $Y_j$. In particular $\bs{Z}^{(1)} = \bs{\mathcal{X}}^{(N)}$ and $\bs{Z}^{(N)} = (\bs{Y}_1,\dots,\bs{Y}_N)$.

For the reader's convenience, we shall restate in the appendix some useful results from \cite{NourdinDeya} and \cite{Kargin}, to which we will often refer to (for instance, the rule for the binomial expansion for free random variables).
\newline
Let us fix some further notation. If $n\geq 2$, for any integer $N$ and $j=1,\dots,n$, consider a kernel $f_N^{(j)}:[N]^{d}\rightarrow \mathbb{R}$ that is mirror symmetric, vanishing on diagonals and with unit variance, as well as the homogeneous polynomial in the non-commuting variables $x_1,\dots,x_N$:
\begin{equation}
\label{QNj}
Q_N^{(j)}(x_1,\dots, x_N) = \sum_{i_1,\dots,i_{d}=1}^{N}f_N^{(j)}(i_1,\dots,i_{d})x_{i_1}\cdots x_{i_{d}}.
\end{equation}
The invariance principle we are interested in concerns vectors of the type $\big(Q_N^{(1)}(\bs{\mathcal{X}}^{(N)}),\dots, Q_N^{(n)}(\bs{\mathcal{X}}^{(N)})\big)$, where \mbox{$\mathbf{X} = \{X_i\}_i$} is a sequence of freely independent centered random variables, with unit variance, belonging to the fixed $W^{\ast}$-probability space $(\mathcal{A},\varphi)$, and $\bs{\mathcal{X}}^{(N)}$ is defined as in (\ref{ensemble}).

The asymptotic behaviour of such a vector will be controlled by means of the influence functions (as defined in (\ref{Influence})). In particular, we will extensively use the quantities $\tau_N^{(j)} = \max\limits_{i=1,\dots,N}\mathrm{Inf}_{i}(f_N^{(j)})$, for $j=1,\dots,n$.

Note that the our multidimensional invariance principle will be stated for Chebyshev sums with mirror symmetric kernels, but to derive from it the universal laws we will have to deal only with fully symmetric kernels.

Keeping the above notation, the forthcoming Theorem \ref{Multiinvariance2} states an invariance principle for vectors of homogeneous sums with low influence, whose proof is given in detail in the last section. 
Note that the bound we provide is of the same nature as the ones given in \cite[Theorem 4.1]{NourdinPeccati4}.

\begin{teo}
\label{Multiinvariance2}
If $d\geq 1$, let $h_1,\dots,h_d$ be positive integers with $h_i = h_{d-i+1}$ for $i=1,\dots,\lfloor \frac{d}{2}\rfloor$ (if $d \geq 2$).
Let $\bs{X} = \{X_i\}_i$ be a sequence of freely independent random variables such that $U_{h_j}(X_i)$ is centered and has unit variance. Consider the vector of Chebyshev sums  $(Q_N^{(1)}(\bs{\mathcal{X}}^{(N)}),\dots,Q_N^{(n)}(\bs{\mathcal{X}}^{(N)}))$ with $Q_N^{(j)}$ of degree $d$ and with mirror symmetric kernels $f_N^{(j)}: [N]^{d}\rightarrow \mathbb{R}$, vanishing on diagonals and having unit variance. Let $\bs{Y}=\{Y_j\}_j$ be a sequence of freely independent centered random variables, with unit variance and freely independent of $\bs{X}$. Assume further that $\bs{X}$ and $\bs{Y}$ are both sequences of identically distributed elements or elements with uniformly bounded moments. Then, for every integer $k\geq 1$ and for every choice of nonnegative integers $m_{1,s},\dots,m_{n,s}$, for $s=1,\dots,k$:
\begin{align}
\varphi\bigg(\prod_{s=1}^{k}\big(Q_N^{(1)}(\bs{\mathcal{X}}^{(N)})\big)^{m_{1,s}}\cdots \big(Q_N^{(n)}(\bs{\mathcal{X}}^{(N)})\big)^{m_{n,s}}\bigg) &-\varphi\bigg(\prod_{s=1}^{k}\big(Q_N^{(1)}(\bs{Y})\big)^{m_{1,s}}\cdots \big(Q_N^{(n)}(\bs{Y})\big)^{m_{n,s}}\bigg) = \\
&= \mathcal{O}\big(\max\limits_{j=1,\dots,n}(\tau_{N}^{(j)})^{\frac{1}{2}}\big)
\end{align}

\end{teo}

\begin{oss}
The complete proof of Theorem \ref{Multiinvariance2} is given in Section 4.2. Here we wish to give some intuition about its structure. As anticipated, the key of our approach consists in considering the vectors
 $\bs{Z}^{(i)} = (\bs{Y}_1,\dots,\bs{Y}_{i-1},\bs{\mathcal{X}}_i,\dots,\bs{\mathcal{X}_N})$, with $\bs{Y}_i = \underbrace{(Y_i,\dots,Y_i)}_{d \text{ times}}$ and $\bs{\mathcal{X}}_i= (U_{h_1}(X_i),\dots,U_{h_d}(X_i))$ \footnote{We drop the dependence on $N$ to simplify the notation.}. As a consequence, one can write:
\begin{align}
\label{differenceBIS}
&\varphi\bigg(\prod_{s=1}^{k}\big(Q_N^{(1)}(\bs{\mathcal{X}}^{(N)})\big)^{m_{1,s}}\cdots \big(Q_N^{(n)}(\bs{\mathcal{X}}^{(N)})\big)^{m_{n,s}}\bigg) -\varphi\bigg(\prod_{s=1}^{k}\big(Q_N^{(1)}(\bs{Y})\big)^{m_{1,s}}\cdots \big(Q_N^{(n)}(\bs{Y})\big)^{m_{n,s}}\bigg) \\ 
&= \sum_{i=1}^{N}
\varphi\bigg(\prod_{s=1}^{k}\big(Q_N^{(1)}(\bs{Z}^{(i)})\big)^{m_{1,s}}\cdots \big(Q_N^{(n)}(\bs{Z}^{(i)})\big)^{m_{n,s}}\bigg) -\varphi\bigg(\prod_{s=1}^{k}\big(Q_N^{(1)}(\bs{Z}^{(i+1)})\big)^{m_{1,s}}\cdots \big(Q_N^{(n)}(\bs{Z}^{(i+1)})\big)^{m_{n,s}}\bigg)\nonumber  \\ 
&= \sum_{i=1}^{N}
\varphi\bigg(\prod_{s=1}^{k}\big(W_1^{(i)}+ V_1^{(i)}(\bs{\mathcal{X}}_i)\big)^{m_{1,s}}\cdots \big(W_n^{(i)}+ V_n^{(i)}(\bs{\mathcal{X}}_i)\big)^{m_{n,s}}\bigg) \nonumber\\ 
&-\varphi\bigg(\prod_{s=1}^{k}\big(W_1^{(i)}+ V_1^{(i)}(\bs{Y}_i))\big)^{m_{1,s}}\cdots \big(W_n^{(i)}+ V_n^{(i)}(\bs{Y}_i)\big)^{m_{n,s}}\bigg),\nonumber
\end{align}
where, for every $j=1,\dots,n$, we set $Q_N^{(j)}(\bs{Z}^{(i)}) = W_j^{(i)} + V_j^{(i)}(\bs{\mathcal{X}}_i)$, with $W_j^{(i)}$, $V_j^{(i)}(\bs{\mathcal{X}}_i)$ self-adjoint operators defined by:
\begin{equation}
\label{WN}
W_j^{(i)} = \sum_{i_1,\dots,i_d \in [N]\setminus \{i\}}f_N^{(j)}(i_1,\dots,i_d)Z_{i_1,1}^{(i)}\cdots Z_{i_d,d}^{(i)}
\end{equation}
(that is, $W_j^{(i)}$ is obtained by gathering together the summands where no $U_{h_p}(X_i)$ appears), and
\begin{equation}
\label{VN}
V_j^{(i)}(\bs{\mathcal{X}}_i) = \sum_{l=1}^{d}\sum_{\substack{i_1,\dots,i_{d-1}\in\\ [N]\setminus \{i\}}}f_N^{(j)}(i_1,\dots,i_{l-1},i,i_{l},\dots,i_{d-1})Z_{i_1,1}^{(i)}\cdots Z_{i_{l-1},l-1}^{(i)} U_{h_l}(X_i) Z_{i_{l}, l+1}^{(i)}\cdots Z_{i_{d-1},d}^{(i)}.
\end{equation}
Similarly, we set:
\begin{equation}
V_j^{(i)}(\bs{Y}_i) = \sum_{l=1}^{d}\sum_{\substack{i_1,\dots,i_{d-1}\\ \in [N]\setminus \{i\}}}f_N^{(j)}(i_1,\dots,i_{l-1},i,i_{l},\dots,i_{d-1})Z_{i_1,1}^{(i)}\cdots Z_{i_{l-1},l-1}^{(i)} Y_i Z_{i_{l}, l+1}^{(i)}\cdots Z_{i_{d-1},d}^{(i)}.
\end{equation}

The conclusion is then obtained by showing that either the summands in (\ref{differenceBIS}) cancel out, either they are zero, either they are of the order of $\max\limits_{j=1,\dots,n}(\tau_N^{(j)})^{\frac{1}{2}}$.

In the next example, we will show how one  can control the expression (\ref{differenceBIS}) for a precise choice of parameters.
\end{oss}

\begin{ex}

Consider $d=3, k=1, n=2, m_{1,1}=2, m_{2,1}=1$. Then, for every fixed $i=1,\dots,N$, in the expansion for
$$ \varphi\big( (W_1^{(i)} + V_{1}^{(i)}(\bs{\mathcal{X}}^{(N)}))^2  (W_2^{(i)} + V_{2}^{(i)}(\bs{\mathcal{X}}^{(N)}))\big) $$ we will have the sum of the following 8 items:
\begin{enumerate}
\item $\varphi\big((W_1^{(i)})^2 W_2^{(i)}\big)$, that will be canceled out in the difference (\ref{differenceBIS}) with the same expectation coming from $\varphi\big( (W_1^{(i)} + V_{1}^{(i)}(\bs{Y}))^2  (W_2^{(i)} + V_{2}^{(i)}(\bs{Y}))\big)$;
\item $\varphi \big((W_1^{(i)})^2 V_2^{(i)}(\bs{\mathcal{X}}^{(N)})\big)$;
\item $\varphi \big(W_1^{(i)} V_1^{(i)}(\bs{\mathcal{X}}^{(N)}) W_2^{(i)}\big) $;
\item $\varphi \big(W_1^{(i)} V_1^{(i)}(\bs{\mathcal{X}}^{(N)}) V_2^{(i)}(\bs{\mathcal{X}}^{(N)})\big) $;
\item $\varphi \big( V_1^{(i)}(\bs{\mathcal{X}}^{(N)}) W_1^{(i)} W_2^{(i)}\big) $;
\item $\varphi \big( V_1^{(i)}(\bs{\mathcal{X}}^{(N)}) W_1^{(i)}  V_2^{(i)}(\bs{\mathcal{X}}^{(N)})\big) $;
\item $\varphi \big( V_1^{(i)}(\bs{\mathcal{X}}^{(N)})^2 W_2^{(i)}  \big) $;
\item $ \varphi \big( V_1^{(i)}(\bs{\mathcal{X}}^{(N)})^2 V_2^{(i)}(\bs{\mathcal{X}}^{(N)}) \big).$
\end{enumerate}
It is easily seen by calculation that the items $2,3$, and $5$ are always zero.
The items 4,6, and 7, are sums of terms that are either zero or cancel with the corresponding terms in 
$\varphi\big( (W_1^{(i)} + V_{1}^{(i)}(\bs{\mathcal{X}}^{(N)}))^2  (W_2^{(i)} + V_{2}^{(i)}(\bs{\mathcal{X}}^{(N)}))\big)$. For instance, if we consider the fourth item, we will have (among other summands that equal zero):
$$ \sum_{\substack{i_1,i_2,i_3 \neq i \\ k_1,k_2 \neq i, l_1,l_2 \neq i} }f_N^{(1)}(i_1,i_2,i_3)f_N^{(1)}(k_1,k_2,i)f_N^{(2)}(i,l_1,l_2)\varphi\big( Z_{i_1}Z_{i_2}Z_{i_3} Z_{k_1}Z_{k_2}U_{h_3}(X_i)U_{h_1}(X_i)Z_{l_1}Z_{l_2}   \big),$$
which becomes (remember that $h_1=h_3$):
\begin{align*}
 \sum_{\substack{i_1,i_2,i_3 \neq i \\ k_1,k_2 \neq i, l_1,l_2 \neq i}}&f_N^{(1)}(i_1,i_2,i_3)f_N^{(1)}(k_1,k_2,i)f_N^{(2)}(i,l_1,l_2)\varphi\big( Z_{i_1}Z_{i_2}Z_{i_3} Z_{k_1}Z_{k_2}U_{h_1}(X_i)^2 Z_{l_1}Z_{l_2}   \big) \\
&= \sum_{i_1,i_2,i_3 \neq i}f_N^{(1)}(i_1,i_2,i_3)f_N^{(1)}(i_3,i_2,i)f_N^{(2)}(i,i_2,i_1)\varphi\big( Z_{i_2}^3 \big). \numberthis \label{es1}
\end{align*}
On the other hand, the same computations yield 
\begin{align*}
 \sum_{\substack{i_1,i_2,i_3 \neq i \\ k_1,k_2 \neq i, l_1,l_2 \neq i}}& f_N^{(1)}(i_1,i_2,i_3)f_N^{(1)}(k_1,k_2,i)f_N^{(2)}(i,l_1,l_2)\varphi\big( Z_{i_1}Z_{i_2}Z_{i_3} Z_{k_1}Z_{k_2}Y_i^2 Z_{l_1}Z_{l_2}   \big)\\
 &= \sum_{i_1,i_2,i_3 \neq i}f_N^{(1)}(i_1,i_2,i_3)f_N^{(1)}(i_3,i_2,i)f_N^{(2)}(i,i_2,i_1)\varphi\big( Z_{i_2}^3 \big), \numberthis \label{es2}
\end{align*}
so that (\ref{es1}) and (\ref{es2}) cancel each other in (\ref{differenceBIS}). 
Note that by the traciality of the state $\varphi$, the computations required for the items 4,6, and 7, are similar, the only difference being in the occuring kernels.

The case to pay more attention to is the item 8. In this case, a priori, we cannot say anything about its value, because it may depend on the distribution of $U_{h_j}(X_i)$. 
Indeed, by linearity, traciality property of $\varphi$ and the rule of free independence, the only non trivial case to be considered is:
$$ \sum_{\substack{i_1,i_2 \neq i \\ l_1,l_2 \neq i \\ k_1,k_2 \neq i} } f_N^{(1)}(i_1,i,i_2)f_N^{(1)}(l_1,i,l_2)f_N^{(2)}(k_1,i,k_2)\varphi\big( Z_{i_1}U_{h_2}(X_i)Z_{i_2}Z_{l_1}U_{h_2}(X_i)Z_{l_2} Z_{k_1}U_{h_2}(X_i) Z_{k_2}\big)$$
when $i_2=l_1, l_2=k_2, k_2=i_1$. Indeed, in this case,
$$\varphi\big( Z_{i_1}U_{h_2}(X_i)Z_{i_2}Z_{l_1}U_{h_2}(X_i)Z_{l_2} Z_{k_1}U_{h_2}(X_i) Z_{k_2}\big) = \varphi\big(U_{h_2}(X_i)^3\big). $$
Similarly, replacing $\bs{\mathcal{X}}^{(N)}$ with $\bs{Y}$, we will have 
 $$\varphi\big( Z_{i_1}Y_i(X_i)Z_{i_2}Z_{l_1}Y_i Z_{l_2} Z_{k_1} Y_i Z_{k_2}\big) = \varphi\big(Y_i^3\big).$$

\end{ex}

\begin{ex}
Consider the case $n=d=2$, and the kernels:
\begin{enumerate}
\item 
$$ f_N^{(1)}(i,j) = \begin{cases}
\dfrac{1}{\sqrt{2N-2}} & \text{ if } i \neq j, i = 1 \vee j=1 \\
0 & \text{otherwise};
\end{cases}
$$
\item
$$ 
f_N^{(2)}(i,j)=
\begin{cases}
0 & \text{ if } i =j \\
\dfrac{1}{\sqrt{N(N-1)}} & \text{ if } i \neq j \; ;
\end{cases}
$$
\item 
$$ 
f_N^{(3)}(i,j) =
\begin{cases}
\dfrac{1}{\sqrt{(N-1)(N-2)}} & \text{ if } i\neq j \text{ and } i, j \neq 1 \\
0 & \text{ otherwise };
\end{cases}
$$
\end{enumerate}
Note that $\|f_N^{(j)}\|^2 = 1$ for all $j=1,2,3$. Simple computations yield that 
\begin{enumerate}
\item $\mathrm{Inf}_1(f_N^{(1)}) = 1$ and $\mathrm{Inf}_j(f_N^{(1)}) = \dfrac{1}{N-1} $ for $j=2,\dots,N$;
\item $\mathrm{Inf}_i(f_N^{(2)}) = \dfrac{2}{N}$ for every $i=1,\dots,N$,
\item $\mathrm{Inf}_1(f_N^{(3)}) = 0$, and $\mathrm{Inf}_j(f_N^{(3)}) = \dfrac{2}{N-1}$ for all $j=2,\dots,N$,
\end{enumerate}
which in turn imply that $\tau_N^{(1)} = 1$, $\tau_N^{(2)} = \dfrac{2}{N}$ and $\tau_N^{(3)} = \dfrac{2}{N-1}$. Therefore, for Chebyshev sums $Q_N^{(1)}$, $Q_N^{(2)}$, and $Q_N^{(3)}$ with kernels $f_N^{(1)},f_N^{(2)}, f_N^{(3)}$ respectively, for any $N \geq 1$, one has:
\begin{equation}
\varphi\bigg(\prod_{s=1}^{k}\big(Q_N^{(1)}(\bs{\mathcal{X}}^{(N)})\big)^{m_{1,s}} \big(Q_N^{(2)}(\bs{\mathcal{X}}^{(N)})\big)^{m_{2,s}}\bigg) -\varphi\bigg(\prod_{s=1}^{k}\big(Q_N^{(1)}(\bs{Y})\big)^{m_{1,s}}\big(Q_N^{(2)}(\bs{Y})\big)^{m_{2,s}}\bigg) 
= \mathcal{O}\big(1\big)
\end{equation}
and so we cannot deduce any universal behaviour, while
\begin{equation}
\varphi\bigg(\prod_{s=1}^{k}\big(Q_N^{(2)}(\bs{\mathcal{X}}^{(N)})\big)^{m_{1,s}} \big(Q_N^{(3)}(\bs{\mathcal{X}}^{(N)})\big)^{m_{2,s}}\bigg) -\varphi\bigg(\prod_{s=1}^{k}\big(Q_N^{(2)}(\bs{Y})\big)^{m_{1,s}}\big(Q_N^{(3)}(\bs{Y})\big)^{m_{2,s}}\bigg) 
= \mathcal{O}\big(\dfrac{1}{\sqrt{N-1}}\big).
\end{equation}
\end{ex}

\subsection{Convergence results}

The results of this subsection are not based on the Lindeberg principle. Indeed, the forthcoming Theorems \ref{ConvChebySemicircular} and \ref{ConvFreePoissonPARI} aim to state the Fourth moment Theorem for Chebyshev sums  in terms of the contraction operators, for semicircular and free Poisson limit respectively (see \cite[Theorems 1.3 and 1.6]{NourdinPeccatiSpeicher}  and  \cite{NourdinPeccati}). The following auxiliary lemma (whose proof requires only simple computations), is inspired by Proposition $4.1$ in \cite{PeccatiZeng} and will be useful in the sequel.
\begin{lemma}
\label{stimacontr} 
Let $d\geq 2$ and $f_N:[N]^{d}\rightarrow \mathbb{R}$. For every $q=1,\dots,d-1$, we have:
\begin{eqnarray}
\|f_N \stackrel{q}{\smallfrown} f_N \| &\geq & \|f_N \star_{q+1}^{q} f_N \|, \nonumber \\
\|f_N \stackrel{1}{\smallfrown} f_N\| &\geq &\|f_N \star_{1}^{0}f_N \|.\nonumber
\end{eqnarray}
\end{lemma}

\begin{teo}
\label{ConvChebySemicircular}
Let $\{S_i\}_i$ be a sequence of freely independent standard semicircular elements in $(\mathcal{A},\varphi)$. If $d\geq 2$, fix integers $h_1,\dots,h_d \geq 1$, with $h_i = h_{d-i+1}$ for all $i=1,\dots,\lfloor \frac{d}{2} \rfloor$ and let $Q_N^{(\bs{h})}(f_N;\cdot)$ be the corresponding Chebyshev sum, as in (\ref{SumCheby2}). The following conditions are equivalent as $N$ goes to infinity:
\begin{enumerate}
\item $Q_N^{(\bs{h})}(f_N;S_1,\dots,S_N)$ converges in law to a standard semicircular random variable $\mathcal{S} \sim \mathcal{S}(0,1)$, freely independent of $\{S_i\}_i$;
\item for every $q=1,\dots, d-1$, $\lim\limits_{N\rightarrow \infty}\|f_N \stackrel{q}{\smallfrown} f_N \| =0$.
\end{enumerate}
\end{teo}

\begin{proof}
Assume that $1$ holds. Then, by virtue of the identity (\ref{Integral}), it is sufficient to remark that \linebreak \mbox{$Q_N^{(\bs{h})}(f_N;S_1,\dots,S_N) = I_{m}^{S}(k_N)$,} with $m=h_1 + \cdots + h_d$, and $k_N$ the kernel given by (\ref{kN1}). At this point, the fourth moment Theorem (see Theorem \ref{FourthMoment}) guarantees the vanishing of the non trivial contractions $\|k_N \stackrel{r}{\smallfrown}k_N \|$, for $r=1,\dots,m-1$. Thanks to Proposition \ref{contraction5}, this in turn implies that the norm $\|f_N \stackrel{q}{\smallfrown} f_N\|$ vanishes in the limit for every $q=1,\dots,d-1$.

To show the converse, it is sufficient to repeat the same reasoning but keeping in mind also the Lemma \ref{stimacontr}.
$ $
\end{proof}

\begin{teo}
\label{ConvFreePoissonPARI}
Let $\{S_i\}_i$ be a sequence of freely independent standard semicircular elements in $(\mathcal{A},\varphi)$, and $Q_N^{(\bs{h})}(f_N;\cdot)$ be a Chebyshev sum as defined in (\ref{SumCheby2}), with both $d$ and $h_1+\cdots +h_d$ even integers. Let $Z(\lambda)$ be a (centered) free Poisson distributed random variable with parameter $\lambda > 0$, freely independent of $\{S_i\}_{i}$, such that:
\begin{equation}
\lim_{N \rightarrow \infty}\varphi\bigg(\big(Q_N^{(\bs{h})}(f_N;S_1,\dots,S_N)\big)^{2}\bigg) = \lambda.
\end{equation}
The following conditions are equivalent:
\begin{itemize}
\item[(i)] $Q_N^{(\bs{h})}(f_N;S_1,\dots,S_N)$ converges in law to $Z(\lambda)$;
\item[(ii)] 
\begin{enumerate}
\item for every $q=1,\dots, d-1$, $q \neq \dfrac{d}{2}$, $\lim\limits_{N \rightarrow \infty}\|f_N \stackrel{q}{\smallfrown} f_N \| = 0$;
\item $\lim\limits_{N \rightarrow \infty}\|f_N \star_{\frac{d}{2}+1}^{\frac{d}{2}}f_N \|= 0 $, and $\lim\limits_{N\rightarrow \infty}\|f_N \stackrel{\frac{d}{2}}{\smallfrown}f_N - f_N\|=0$.
\end{enumerate}
\end{itemize}
\end{teo}

\begin{proof}
Again, $Q_N^{(\bs{h})}(f_N;S_1,\dots,S_N) = I_m^{S}(k_N)$, with $m=h_1 + \cdots + h_d$, and $k_N$ as in (\ref{kN1}). Now simply apply Theorem \ref{ConvWignerToFreePoisson}, together with Proposition \ref{contraction5}, Proposition \ref{contraction6} and Lemma \ref{stimacontr}.
\end{proof}

\begin{oss}[\textbf{On the parity of} $\mathbf{d}$.]
Let us remark that for the convergence of a Chebyshev sum towards the free Poisson law, it is not sufficient that the sum of the orders $h_1,\dots,h_d$ is even. Indeed, if $d$ is odd and we assume that $Q_N^{(\bs{h})}(f_N;S_1,\dots,S_N) = I_m^{S}(k_N)$ converges to $Z(\lambda)$, then we would have $\|k_N \stackrel{r}{\smallfrown} k_N \|$ vanishing in the limit for every $r=1,\dots,m-1$, $r\neq \dfrac{m}{2}$. In particular, if $r=h_1 + \cdots + h_{\frac{d-1}{2}}$, $\|k_N  \stackrel{r}{\smallfrown} k_N\| = \| f_N \stackrel{\frac{d-1}{2}}{\smallfrown} f_N \|$ would tend to zero. By virtue of Lemma \ref{stimacontr}, this would imply that $\|f_N \star_{\frac{d+1}{2}}^{\frac{d-1}{2}} f_N \|$ tends to zero. 
On the other hand, $\| f_N \star_{\frac{d+1}{2}}^{\frac{d-1}{2}} f_N \| = \| k_N \stackrel{\frac{m}{2}}{\smallfrown} k_N\|$, that should not tend to zero, yielding a contradiction.

Henceforth we are able to establish conditions for the convergence of a Chebyshev sum towards the free Poisson law only if both $d$ and $h_1+\cdots +h_d$ are even integers.
$ $
\end{oss}

\begin{oss}
From Theorem \ref{ConvChebySemicircular} and Theorem \ref{ConvFreePoissonPARI} with $h_j = 2$ for every $j=1\dots,d$ and with $d$ even, since $U_2(S) \stackrel{\text{law}}{=} Z(1)$, we can deduce explicit conditions for the convergence of a homogeneouos sum as in (\ref{QN}) based on a sequence $\{Z_i\}_i$ of freely independent and centered random variables with free Poisson distribution of parameter $1$, towards the semicircular law (generalizing to the free setting the findings of \cite{PeccatiZeng}) and the free Poisson law (if $d$ is even).
See moreover \cite{Solesne}, Theorem 4.1, for a general fourth moment statement for Free Poisson multiple integrals.
\end{oss}

\subsection{Universality results}

As straightforward consequences of the invariance principle stated in Section 3.1, we will derive possible universal limit laws for vectors of homogeneous sums.  They will have the same nature as the Theorem 7.2 in \cite{NourdinPeccati4}, where the authors prove that the normal distribution is universal for vectors of homogeneous sums with respect to multivariate Gaussian approximation.

To this aim, let the above notation for vectors of Chebyshev sums prevail, except that from now on we shall assume that their kernels $f_N$ are fully symmetric functions. 
In particular, if $d\geq 2$, consider fixed integers $h_1,\dots,h_d$ with $h_i = h_{d-i+1}$ for $i=1,\dots,\lfloor \frac{d}{2}\rfloor$, and a sequence $\bs{X}=\{X_i\}$ of freely independent centered random variables such that $\varphi(U_{h_j}(X_i)) = 0$ and $\varphi(U_{h_j}(X_i)^2) = 1$ for all $j=1,\dots,d$ and for every $i$.
Recall that if $\bs{\mathcal{X}}^{(N)} = (\bs{\mathcal{X}}_1,\dots,\bs{\mathcal{X}}_n)$, with $\bs{\mathcal{X}}_i = (U_{h_1}(X_i),\dots,U_{h_d}(X_i))$ for all $i$, then:
$$ Q_N(\bs{\mathcal{X}^{(N)}}) = \sum_{i_1,\dots,i_d=1}^{N} f_N(i_1,\dots,i_d)U_{h_1}(X_{i_1})\cdots U_{h_d}(X_{i_d}).$$

Let us denote by $\mathcal{NC}_2([n])$ the set of all the \textit{non-crossing pairings} of $[n] = \{1,2,\dots,n\}$, that is the set of all non-crossing partitions of the set $[n]$ where each block has exactly two elements. Of course,  $\mathcal{NC}_2([n])$ is empty if $n$ is odd, while it has $C_{\frac{n}{2}}$ elements if $n$ is even (see \cite{Speicher}). If $s_1,\dots, s_n$ are standard semicircular elements, with covariance $\varphi(s_i s_j) = C_{i,j}$ such that the matrix $C = (C_{i,j})$ is positive definite, the joint moments of $s_1,\dots,s_n$ are completely determined by $C$ according to the following Wick-type formula (see \cite{Speicher}): for every $m$ and every integers $i_1,\dots,i_m \in [n]$,
$$ \varphi(s_{i_1}s_{i_2}\cdots s_{i_n}) = \sum_{\pi \in \mathcal{NC}_2([m])}\prod_{(r,p)\in \pi}\varphi(s_{i_r}s_{i_p}).$$

\begin{teo}
\label{MultUniversality}
Let $d\geq 2$, and let $\mathbf{S} =\{S_i\}_i$ be a sequence of freely independent standard semicircular random variables. Let $(s_1,\dots, s_n)$ be a standard semicircular vector, with covariance $\varphi(s_i s_j) = C_{i,j}$ for every $i,j=1,\dots,n$. Suppose moreover that:
$$ \lim_{N \rightarrow \infty} \varphi\big(Q_N^{(i)}(\bs{\mathcal{S}}^{(N)}) Q_N^{(j)}(\bs{\mathcal{S}}^{(N)})\big) = C_{i,j},$$
with $\bs{\mathcal{S}}^{(N)} = (\bs{\mathcal{S}}_1, ,\dots,\bs{\mathcal{S}}_N)$, and $\bs{\mathcal{S}}_j = (U_{h_1}(S_j),\dots,U_{h_d}(S_j))$.
Then the following assertions are equivalent as $N$ goes to infinity:
\begin{itemize}
\item[(i)]  $Q_N^{(j)}(\bs{\mathcal{S}}^{(N)}) \stackrel{\text{law}}{\longrightarrow} \mathcal{S}(0,C_{j,j})$, for every $j=1,\dots,n$;
\item[(ii)] $(Q_N^{(1)}(\mathbf{X}),\dots, Q_N^{(n)}(\mathbf{X})) \stackrel{\text{ law }}{\longrightarrow }(s_1,\dots,s_n)$ for every sequence $\mathbf{X} = \{X_i\}_i$ of freely independent and identically distributed centered random variables with unit variance.
\end{itemize}

\end{teo}

\begin{proof}
$ $
\begin{itemize} 
\item[(i) $\Rightarrow$ (ii)]
Thanks to \cite[Theorem 1.3]{NouSpeiPec}, the hypotheses $Q_N^{(j)}(\mathcal{S}^{(N)})\stackrel{ \text{law}}{\longrightarrow} \mathcal{S}(0,C_{j,j})$ for all $j=1,\dots,n$ are equivalent to the joint convergence $(Q_N^{(1)}(\bs{\mathcal{S}}^{(N)}),\dots, Q_N^{(n)}(\bs{\mathcal{S}}^{(N)})) \stackrel{\text{ law }}{\longrightarrow }(s_1,\dots,s_n)$. 
In particular, we have that \mbox{$\|f_N^{(j)} \stackrel{d -1}{\smallfrown} f_N^{(j)} \| \rightarrow 0$} for every $j$, and therefore it follows that $\tau_N^{(j)} \rightarrow 0$ for $j=1,\dots,n$ as $N$ goes to infinity (see  Lemma \ref{magg}). This in turn trivially implies that \linebreak $\max\limits_{j=1,\dots,n}\tau_N^{(j)} \rightarrow 0$ and the conclusion follows by Theorem \ref{Multiinvariance2}.

\item[(ii) $\Rightarrow$ (i)] In particular we have  $(Q_N^{(1)}(\mathbf{S}),\dots, Q_N^{(n)}(\mathbf{S})) \stackrel{\text{ law }}{\longrightarrow }(s_1,\dots,s_n)$, with $\mathbf{S}$ denoting a sequence of freely independent standard semicirular elements. But this implies that $\max_{j=1,\dots,n}\tau_N^{(j)} \rightarrow 0$, yielding first $(Q_N^{(1)}(\bs{\mathcal{S}}^{(N)}), \dots, Q_N^{(n)}(\bs{\mathcal{S}}^{(N)})) \stackrel{\text{law}}{\longrightarrow} (s_1,\dots,s_n)$ by virtue of Theorem \ref{Multiinvariance2}, and then the desired componentwise convergence.

\end{itemize}
\end{proof}

By very similar arguments, assuming that $d$ is even and by keeping in mind in particular the relation (\ref{magg2}), it is possible to give immediate proofs of the following statement concerning free Poisson approximation for vectors of Chebyshev sums.

\begin{teo}
\label{Poiss2}
Let $d\geq 2$ be even. Let $\mathbf{S} =\{S_i\}_i$ be a sequence of freely independent standard semicircular random variables. Let $s_1,\dots,s_n$ be standard semicircular elements, with $\varphi(s_i s_j) = C_{i,j}$, and set $z_j= s_j^{2}-1$, so that $z_j$ is a centered free Poisson random variable with parameter $1$.
Assume that $Q_N^{(j)}$ is a homogeneous sum of even degree $d$ and assume that $h_1+\cdots+ h_d$ is  even. If $\bs{\mathcal{S}}^{(N)} = (\bs{\mathcal{S}}_1, \dots,\bs{\mathcal{S}}_{N})$, $\bs{\mathcal{S}}_j = (U_{h_1}(S_j),\dots,U_{h_d}(S_j))$, the following assertions are equivalent as $N$ goes to infinity:
\begin{itemize}
\item[(i)] $(Q_N^{(1)}(\bs{\mathcal{S}}^{(N)}),\dots, Q_N^{(n)}(\bs{\mathcal{S}}^{(N)})) \stackrel{\text{ law }}{\longrightarrow }(z_1,\dots,z_n)$;
\item[(ii)] $(Q_N^{(1)}(\mathbf{X}),\dots, Q_N^{(n)}(\mathbf{X})) \stackrel{\text{ law }}{\longrightarrow }(z_1,\dots,z_n)$ for every sequence $\mathbf{X} = \{X_i\}_i$ of freely independent and identically distributed centered random variables with unit variance.
\end{itemize}
\end{teo}

\begin{oss}
So far, it is known that componentwise convergence of multiple Wigner integrals towards the semicircular law implies the joint convergence (see \cite[Theorem 1.3]{NouSpeiPec}), but similar results are still missing for Free Poisson approximation. This is the reason why, in Theorem \ref{Poiss2}, we assume the joint convergence of the vector $(Q_N^{(1)}(\bs{\mathcal{S}}^{(N)}),\dots, Q_N^{(n)}(\bs{\mathcal{S}}^{(N)}))$.
\end{oss}

\begin{oss}
If we set $h_j =h$ for all $j=1,\dots,d$, the previous universality results state that sequences of the type $\{U_h(S_i)\}_i$ (belonging to the $h$-th Wigner Chaos) behave universally (for vectors of homogeneous sums) with respect to both semicircular and free Poisson approximation (if $d$ is even), generalizing the universality results established in \cite[Theorem 1.4]{NourdinDeya}, corresponding to the case $h=1$.

In particular, if $h =2$, the corresponding universality statements concerns vectors of homogeneous sums based on a sequence of centered free Poisson random variables of parameter $1$, with respect to both semicircular and free Poisson approximation (when $d$ is an even integer).
\end{oss}

\begin{oss}
Since the conditions required to the kernels of $Q_N^{(\bs{h})}(f_N;S_1,\dots,S_N)$ for the convergence towards the semicircular and the free Poisson laws do not depend on the choice of the orders $h_1,\dots,h_d$, we can conclude that the convergence of a vector of Chebyshev sums of given orders $(h_1,\dots,h_d)$, based on a semicircular system, towards both the semicircular and the free Poisson law, is equivalent to the convergence towards that laws for any other vector of Chebyshev sums with the same kernels. In particular, this holds true for homogeneous sums based on the $h$-th Chebyshev polynomial, for a given $h\geq 1$. We are going to make explicit these remarks only in the one dimensional case for notational convenience.
\end{oss}

\begin{cor}
Let $Q_N$ be the homogeneous sum defined in (\ref{QN}), with $d\geq 2$ and symmetric kernel, and let $\{S_i\}_i$ be a sequence of freely independent standard semicircular random variables. The following assertions are equivalent as $N$ goes to infinity:
\begin{itemize}
\item there exist integers $h_1,\dots,h_d$, with $h_i = h_{d-i+1}$ for $i= 1,\dots, \lfloor \frac{d}{2}\rfloor$, such that $$ Q_N(U_{h_1}(S_1),\dots,U_{h_{d}}(S_N)) \stackrel{\text{ law }}{\longrightarrow} \mathcal{S} \;;$$
\item for every $k_1,\dots,k_d$ such that $k_i = k_{d-i+1}$ for $i= 1,\dots, \lfloor \frac{d}{2}\rfloor$, 
$$Q_N(U_{k_1}(S_1),\dots,U_{k_d}(S_N)) \stackrel{\text{ law }}{\longrightarrow} \mathcal{S}.$$
\end{itemize}
\end{cor}

\begin{cor}
\label{coroll1}
Let $Q_N$ be the homogeneous polynomial defined in (\ref{QN}), with $d \geq 2$ and with symmetric kernel. The following assertions are equivalent as $N$ goes to infinity:
\begin{itemize}
\item if $\{S_i\}_i$ is a sequence of freely independent standard semicircular random variables, then
 $$Q_N(S_1,\dots,S_N) \stackrel{\text{ law }}{\longrightarrow} \mathcal{S};$$
\item if $\{Z_i\}_i$ is a sequence of freely independent centered random variables with free Poisson distribution of parameter $1$, then $$Q_N(Z_1,\dots,Z_N) \stackrel{\text{ law }}{\longrightarrow} \mathcal{S}.$$
\end{itemize}
\end{cor}

\begin{ex}
As an application of the Corollary \ref{coroll1}, consider the homogeneous sum:
$$ Q_N(x_1,\dots,x_N) = \dfrac{1}{\sqrt{2N-2}}\sum_{i=2}^{N}(x_1 x_i + x_i x_1).$$
As shown in \cite{NourdinDeya} in the first counterexample, $Q_N(S_1,\dots,S_N)$ converges in law to $\dfrac{1}{\sqrt{2}}(S_1 S_2 + S_2 S_1)$, and therefore its limit is neither semicircular nor free Poisson distributed (being Tetilla distributed, see \cite{NourdinDeya2}). Corollary \ref{coroll1} gives the additional information that even $Q_N(Z_1,\dots,Z_N)$ cannot converge towards that laws, nor can any other sequence $\{Q_N(U_h(S_1),\dots,U_h(S_N))\}$, $h\geq 3$. 

Remark that with this counterexample the authors were meant to show that the free Rademacher law ($ \mu = \dfrac{1}{2}\delta_1 + \dfrac{1}{2}\delta_{-1}$) is not universal for homogeneous sums. Indeed, they proved that if $\{X_i\}_i$ is a sequence of freely independent Rademacher random variables, then $Q_N(X_1,\dots,X_N)$ has asymptotically semicircular distribution. This is consistent with the fact that the free Rademacher law is not admissible for any chaotic random variable of the type $U_n(S)$, and it implies in turn that the Tetilla law 
cannot be a universal limit law for homogeneous sums of freely independent random variables.
$ $
\end{ex}

\begin{oss}
Thanks to a careful inspection of all the previous proofs, and by considering the estimate (\ref{magg1}), we can conclude that if $d\geq 2$, and $f_N$ is a fully symmetric kernel satisfying $ \|f_N \stackrel{d-1}{\smallfrown} f_N \| \rightarrow 0$ as $N \rightarrow \infty$, then the limit distribution of $Q_N^{(\bs{h})}(f_N;X_1,\dots,X_N)$ (and in particular that of $Q_N(X_1,\dots,X_N)$ with $Q_N$ as in (\ref{QN})), never depends on the distribution of the sequence $\{X_i\}_i$.
\end{oss}

\begin{paragraph}{About classical universality results}
Let us remark how the invariance principle stated in \cite{Mossel} hides similar results for classical probability spaces. Indeed, consider a probability space $(\Omega, \mathcal{F}, \mathbb{P})$, and let $\{X_i\}_i$ be a sequence of independent random variables on it. If $\{H_n(x)\}_n$ denotes the sequence of the (monic) Hermite polynomials, assume that for fixed integers $n_1,\dots,n_d$, $H_{n_j}(X_i)$ is centered and has unit variance for every $i$ and every $j$, and that the third moments are uniformly bounded, say  $\mathbb{E}[|H_{n_j}(X_i)|^{3}] < B$ for all $i$, in such a way that the systems $\bs{\mathcal{X}}^{(i)} = \{H_{n_1}(X_i),\dots, H_{n_d}(X_i)\}$ are $(2,3,\eta)$-hypercontractive. Under these assumptions, if $\{Y_i\}_i$ denotes another sequence of centered independent random variables, having unit variance, and $(2,3,\eta)$-hypercontractive, for every function $\psi \in \mathcal{C}^{3}(\mathbb{R})$ with uniformly bounded third derivative, it holds true that:
$$|\mathbb{E}[\psi\big(Q_N(\bs{\mathcal{X}}^{N})\big)] - \mathbb{E}[\psi\big(Q_N(Y_1,\dots,Y_N)\big)]| \leq C_{\eta, B,\psi} \;(\tau_N)^{\frac{1}{2}}.$$
In particular, if $\mathcal{H}$ denotes a (separable) Hilbert space, and $\bs{X} = \{X(e): e \in \mathcal{H}\}$ is an isonormal Gaussian process on it, consider $X_j = X(e_j)$ with $\|e_j \| = 1$, so that $X_j \sim \mathcal{N}(0,1)$. It is a standard result that $\dfrac{1}{n!}H_{n}(X_i) = I_{n}^{X}(h_i^{\otimes n})$ is centered, with unit variance, and hypercontractive (see, for instance, \cite{Peccatilibro}). 

If now we consider an orthonormal basis $\{e_j\}_j $ of $\mathcal{H}$, the associated sequence $\{X_i\}_i$ is a sequence of independent standard normal variables, and $Q_N(\bs{\mathcal{X}}^{(N)}) = I_m^{X}(k_N)$, with $k_N$ as in (\ref{kN1}). Here we can apply all the \textit{fourth moment}-type results for the convergence of chaotic random variables towards the Gaussian and the Gamma distributions (\cite{NourdinPeccati2}, Theorem 1.2), and get the corresponding universality results (see \cite{NourdinPeccati4}). In particular, if we choose $n_j= n \geq 1$ for all $j=1,\dots,d$, we can deduce that homogeneous sums based on chaotic random variables of the form $H_n(X_i)$  behave universally with respect to both the Gaussian and the Gamma approximation. Note that $k_N$ is not symmetric in general, but if $\widetilde{k}_N$ denotes its standard symmetrization, then $I_m^{X}(k_N) = I_m^{X}(\widetilde{k_N})$.
\end{paragraph}

\paragraph{Concluding remarks}

All the previous results leave opened the possibility for further generalizations to free stochastic integrals with respect to a free Poisson measure $P$ with intensity measure given by the Lebesgue measure $\mu$. More precisely, consider the kernel:
\begin{equation}
\label{gN}
g_N = \sum_{i_1,\dots,i_d=1}^{N}f_N(i_1,\dots,i_d)e_{i_1}\otimes \cdots \otimes e_{i_d},
\end{equation}
with $e_j = \mathbb{1}_{A_j}$, for $A_j$ measurable set with $\mu(A_j) = 1$. If $Q_N$ denotes the homogeneous sum as in (\ref{QN}), then  $Q_N(Z_1,\dots,Z_N) = I_d^{P}(g_N)$, and therefore we have results of convergence for free Poisson integrals towards semicircular and free Poisson laws for simple kernels. We believe that this approach could be extended to more general kernels, but this investigation is left for further work.

Similarly, we believe that the approach we have proposed could fit the more general framework of the stochastic integration with respect to the $q$-Brownian motion, with the $q$-Hermite polynomials replacing the Chebyshev polynomials. Note that, at least for $q \in [0,1]$ and for symmetric kernels, a fourth moment theorem has been recently established (see \cite{qbrownian}). Again, this line of research is left open for further investigation.

\section{Proofs}

\subsection{Auxiliary statements}

The proofs of the universality results are based on the following upper bounds for $\tau_N = \max\limits_{i=1,\dots,N}\mathrm{Inf}_i(f_N)$.
\begin{lemma}
\label{magg}
Let $d \geq 2$, and let $f_N:[N]^{d}\rightarrow \mathbb{R}$ be a symmetric kernel, vanishing on diagonals. If $d \geq 2$, then
\begin{equation}
\label{magg1}
\|f_N \stackrel{d-1}{\smallfrown}f_N\| \geq \dfrac{1}{d}\tau_N.
\end{equation}
Moreover, if $d=2$, then
\begin{equation}
\label{magg2}
\|f_N \stackrel{1}{\smallfrown}f_N - f_N \| \geq \dfrac{1}{2}\tau_N.
\end{equation}
\end{lemma}

\begin{proof}
By carrying out the same computations as in the proof of the Theorem 1.4 in \cite{NourdinDeya}, if $d > 2$, we obtain the following estimates:
\begin{align*}
\|f_N \stackrel{d-1}{\smallfrown}f_N\|^{2} &= \sum\limits_{i_1,i_2=1}^{N}\bigg(f_N \stackrel{d-1}{\smallfrown} f_N (i_1,i_2)\bigg)^{2} \;\geq \;\sum\limits_{i=1}^{N}\bigg(f_N \stackrel{d-1}{\smallfrown} f_N(i,i)\bigg)^{2} \\
&= \sum\limits_{i=1}^{N}\bigg(\sum_{j_2,\dots,j_d=1}^{N}f_N(i,j_2,\dots,j_d)^{2}\bigg)^{2} \; \geq\; \bigg(\sum_{j_2,\dots,j_d=1}^{N}f_N(i,j_2,\dots,j_d)^{2}\bigg)^{2}
\end{align*}
for every $i=1,\dots, N$, and so, by taking the square root on both sides, we have that for every $i=1,\dots, N$,
$$ \| f_N \stackrel{d-1}{\smallfrown}f_N\|\geq \dfrac{1}{d}\mathrm{Inf}_{i}(f_N),$$
from which $\| f_N \stackrel{d-1}{\smallfrown}f_N\| \geq \frac{1}{d}\max\limits_{i=1,\dots,N}\mathrm{Inf}_{i}(f_N) = \dfrac{1}{d}\tau_{N}$.
In the case $d=2$, to get an upper bound for $\tau_N$, we have to consider a different chain of inequalities, namely:
\begin{align*}
\|f_N \stackrel{1}{\smallfrown}f_N - f_N \|^{2} &= \sum_{i,j=1}^{N}\bigg(f_N \stackrel{1}{\smallfrown} f_N (i,j) - f_N(i,j)\bigg)^{2} \\
&= \sum_{i \neq j=1}^{N}\bigg(f_N \stackrel{1}{\smallfrown} f_N (i,j) - f_N(i,j)\bigg)^{2} + \sum_{i=1}^{N}\big(f_N \stackrel{1}{\smallfrown} f_N (i,i)\big)^{2}   \\
&\geq \sum_{i=1}^{N}\bigg(\sum_{k=1}^{N}f_N(i,k)^{2}\bigg)^{2} \;\geq \bigg(\sum_{k=1}^{N}f_N(i,k)^{2}\bigg)^{2}
\end{align*}
for ever $i=1,\dots, N$, from which $\|f_N \stackrel{1}{\smallfrown}f_N - f_N \| \geq  \dfrac{1}{2}\mathrm{Inf}_i(f_N)$ and in particular $\|f_N \stackrel{1}{\smallfrown}f_N - f_N \| \geq \dfrac{1}{2}\tau_N$.
$ $
\end{proof}

The following lemma is meant to generalize the relations given in the Lemma 3.1 in \cite{NourdinDeya}: the proof follows straightforwardly.
\begin{lemma}
\label{3.1bis}
Let $\{\mathcal{A}_i\}_i$ be a sequence of freely independent unital subalgebras of $\mathcal{A}$, with $(\mathcal{A},\varphi)$ a fixed von Neumann algebra. Let $\mathcal{B}$ be a unital subalgebra of $\mathcal{A}$, freely independent of $\{\mathcal{A}_i\}$. For every $B_1,B_2  \in \mathcal{B}$, and $C_p \in \mathcal{A}_p$, centered and with unit variance,
\begin{enumerate}
\item for every $r, s \geq 0$, and every $p_1,\dots,p_s \in \mathbb{N}$,
$\varphi(C_{p_1}\cdots C_{p_r} B_i C_{p_{r+1}}\cdots C_{p_s}) = 0$ \; ;
\item  if $\mathcal{D}$ is any other unital subalgebra freely independent of $\{\mathcal{A}_i\}$,
for every $0\leq r < s \leq k$, and $m_1,\dots,m_k \in \mathbb{N}$, such that there exists at least one $j=r+1,\dots,s$ with $m_{j} = 1$, and any centered element $Z $ in $\mathcal{D}$ with unit variance,
$$ \varphi(C_{p_1}^{m_1}\cdots C_{p_r}^{m_r}B_1 C_{p_{r+1}}^{m_{r+1}}\cdots C_{p_s}^{m_s} B_2 C_{p_{s+1}}^{m_{s+1}}\cdots C_{p_k}^{m_k}) = \varphi(C_{p_1}^{m_1}\cdots C_{p_r}^{m_r} Z C_{p_{r+1}}^{m_{r+1}}\cdots C_{p_s}^{m_s} Z C_{p_{s+1}}^{m_{s+1}}\cdots C_{p_k}^{m_k}), $$
for every choice of integers $p_1 \neq p_2 \neq \cdots \neq p_{r}$, $p_{r+1}\neq p_{r+2}\neq \cdots \neq p_s$, $p_{s+1}\neq p_{s+2}\neq \cdots \neq p_{k}$;
\item if $B =B_1 = B_2$, then:
$$ \varphi(C_{p_1}^{m_1}\cdots C_{p_r}^{m_r}B C_{p_{r+1}}^{m_{r+1}}\cdots C_{p_s}^{m_s} B C_{p_{s+1}}^{m_{s+1}}\cdots C_{p_k}^{m_k}) = \varphi(C_{p_1}^{m_1}\cdots C_{p_r}^{m_r} Z C_{p_{r+1}}^{m_{r+1}}\cdots C_{p_s}^{m_s} Z C_{p_{s+1}}^{m_{s+1}}\cdots C_{p_k}^{m_k}) $$
for every $r \leq s\leq k, m_j = 0 $ or $m_j \geq 2$ for all $j=r+1,\dots,s$.
\end{enumerate}
$ $
\end{lemma}

For the proof of the Theorem \ref{Multiinvariance2}, we will need the following iterated Cauchy-Schwarz inequality.

\begin{lemma}
\label{algo}
Let $c_1,\dots,c_n$ elements in $\mathcal{A}$. Then:
\begin{enumerate}
\item if $n$ is even:
$$ |\varphi\big(c_1\cdots c_n\big)| \leq \prod_{l=1}^{n}\prod_{s_{j} \in I_l(\bs{c})}\varphi\big((c_l c_l^{\ast})^{2^{s_{j}}}\big)^{2^{-\frac{n}{2}}}\; ,$$
where, for every $l=1,\dots,n$, $I_l(\bs{c})$ is a (multi)set of integers\footnote{We are dealing with multisets because repetitions may occur.} $s_{j}$ such that $\sum\limits_{j}2^{s_{j}} = 2^{\frac{n}{2}-1}$;
\item if $n \geq 3$ is odd:
$$ |\varphi\big(c_1\cdots c_n\big)| \leq \prod_{l=1}^{\frac{n-1}{2}}\prod_{s_j \in I_l(\bs{c})}\varphi\big((c_l c_l^{\ast})^{2^{s_{j}}}\big)^{2^{-\frac{n-1}{2}}} \cdot \prod_{l= \frac{n+1}{2}}^{n}\prod_{s_j \in I_l(\bs{c})}\varphi\big((c_l c_l^{\ast})^{2^{s_{j}}}\big)^{2^{-\frac{n+1}{2}}} , $$
where, for every $l=1,\dots,n$, $I_l(\bs{c})$ is a multiset of integers $s_{j}\geq 0$ such that $\sum\limits_{j}2^{s_{j}} = 2^{\frac{n-3}{2}}$ for $l=1,\dots,\frac{n-1}{2}$, and for $l=\frac{n+1}{2},\dots, n$, $\sum\limits_{j}2^{s_{j}} = 2^{\frac{n-1}{2}}$.
\end{enumerate}
\end{lemma}

\begin{oss}
\label{Osserv}
As made clear in the proof, the multiset $I_l(\bs{c})$ is determined by the rule of association one chooses in order to iteratively apply the Cauchy-Schwarz inequality. For our purposes (i.e. the proof of Theorem \ref{Multiinvariance2}), we do not need to further specify the structure of $I_l(\bs{c})$.
\end{oss}

\begin{ex}
For the sake of clarity, let us first show how the technique of the lemma applies in some simple cases: $n=2,3,4,5$.

\begin{enumerate}
\item[(n=2)] We trivially recover the Cauchy-Schwarz inequality:
$$ |\varphi(c_1 c_2)| \leq \varphi(c_1 c_1^{\ast})^{\frac{1}{2}}\varphi(c_2 c_2^{\ast})^{\frac{1}{2}}.$$
\item[(n=3)] By applying the Cauchy-Schwarz inequality and the tracial property of the state $\varphi$, we obtain:
\begin{align*}
|\varphi(c_1 (c_2 c_3))| &\leq \varphi(c_1 c_1^{\ast})^{\frac{1}{2}} \varphi( c_2 c_3 c_3^{\ast} c_2^{\ast})^{\frac{1}{2}} = \varphi(c_1 c_1^{\ast})^{\frac{1}{2}} \varphi((c_2^{\ast} c_2) (c_3 c_3^{\ast}))^{\frac{1}{2}} \\
&\leq \varphi(c_1 c_1^{\ast})^{\frac{1}{2}} \varphi((c_2 c_2^{\ast})^2)^{\frac{1}{4}} \varphi((c_3 c_3^{\ast})^2)^{\frac{1}{4}},
\end{align*}
so that the conclusion of the lemma is achieved by setting $I_1(\bs{c}) = \{0\}$, $I_2(\bs{c})=I_3(\bs{c})=\{1\}$, in such a way that $2^{0}= 2^{\frac{n-3}{2}}$, and $2 = 2^{\frac{n-1}{2}}$. Moreover, $\frac{1}{4} = 2^{-\frac{n+1}{2}}$, and $\frac{1}{2} = 2^{-\frac{n-1}{2}}$.

Note that, had we started by associating the arguments of $\varphi$ as $\varphi( (c_1 c_2)c_3)$, we would have obtained:
$$ 
|\varphi(c_1 c_2 c_3)| \leq \varphi((c_1 c_1^{\ast})^2)^{\frac{1}{4}} \varphi((c_2 c_2^{\ast})^2)^{\frac{1}{4}} \varphi(c_3 c_3^{\ast})^{\frac{1}{2}},$$
 (see Remark \ref{Osserv}), yielding as multiset $I_1(\bs{c}) = \{1\} = I_2(\bs{c}), I_3(\bs{c}) = \{0\}$.
\item[(n=4)]
\begin{align*}
|\varphi((c_1 c_2) (c_3 c_4))| & \leq \varphi\big( (c_1^{\ast}c_1)(c_2 c_2^{\ast})\big)^{\frac{1}{2}}\varphi\big( (c_3^{\ast}c_3)(c_4 c_4^{\ast})\big)^{\frac{1}{2}} \\
&\leq \varphi\big( (c_1^{\ast}c_1)^{2}\big)^{\frac{1}{4}}\varphi\big( (c_2^{\ast}c_2)^{2}\big)^{\frac{1}{4}}\varphi\big( (c_3^{\ast}c_3)^{2}\big)^{\frac{1}{4}}\varphi\big( (c_4^{\ast}c_4)^{2}\big)^{\frac{1}{4}}
\end{align*}
so that the conclusion of the lemma is achieved by setting $I_l(\bs{c}) = \{1\}$ for $l=1,\dots,4$, with $2= 2^{\frac{n}{2}-1}$, and $\frac{1}{4} = 2^{-\frac{n}{2}}$.

\item[(n=5)]
\begin{align*}
|\varphi((c_1 c_2) (c_3 c_4 c_5))| &\leq \varphi\big( (c_1^{\ast}c_1)(c_2 c_2^{\ast})\big)^{\frac{1}{2}} \varphi\big(((c_3^{\ast}c_3)c_4)((c_5 c_5^{\ast})c_4^{\ast})\big)^{\frac{1}{2}} \\
&\leq \varphi\big( (c_1^{\ast}c_1)^{2}\big)^{\frac{1}{4}}\varphi\big( (c_2^{\ast}c_2)^{2}\big)^{\frac{1}{4}} \varphi\big( (c_3^{\ast}c_3)^{2}(c_4 c_4^{\ast})\big) ^{\frac{1}{4}} \varphi\big( (c_5^{\ast}c_5)^{2}(c_4 c_4^{\ast})\big) ^{\frac{1}{4}} \\
&\leq \varphi\big( (c_1^{\ast}c_1)^{2}\big)^{\frac{1}{4}}\varphi\big( (c_2^{\ast}c_2)^{2}\big)^{\frac{1}{4}}  \varphi\big( (c_3 c_3^{\ast})^4 \big)^{\frac{1}{8}}\varphi\big( (c_4 c_4^{\ast})^2 \big)^{\frac{1}{8}}\varphi\big( (c_5 c_5^{\ast})^4 \big)^{\frac{1}{8}}\varphi\big( (c_4 c_4^{\ast})^2 \big)^{\frac{1}{8}},
\end{align*}
so that the conclusion of the lemma is achieved by setting $I_1(\bs{c})=I_{2}(\bs{c}) =\{1\}$, giving $2 = 2^{\frac{n-3}{2}}$ and $\frac{1}{4} = 2^{-\frac{n-1}{2}}$, and $I_{3}(\bs{c}) =I_{5}(\bs{c})=\{ 2\}$ so that $2^{2} = 2^{\frac{n-1}{2}}$, $I_4(\bs{c})=\{1,1\}$, so that $2+2=2^{\frac{n-1}{2}}$, and $\frac{1}{8} = 2^{-\frac{n+1}{2}}$.
\end{enumerate}

\end{ex}

\begin{proof}
Suppose first that $n$ is even, say $n=2k$, and we proceed by induction on $k$. If $n=2$, then we recover the Cauchy-Schwarz inequality. 

So assume that $k > 1$ and that our statement is true for $n=2h$, for all $h \leq k$. If $n=2(k+1)$, apply the Cauchy-Schwarz inequality in the following way:
\begin{align*}
|\varphi\big((c_1 \cdots c_{k+1})& (c_{k+2}\cdots c_{n})\big)| \leq \varphi\big(c_1 c_2 \cdots c_{k+1}c_{k+1}^{\ast}\cdots c_{2}^{\ast}c_1^{\ast}\big)^{\frac{1}{2}} \varphi\big(c_{k+2}\cdots c_n c_n^{\ast}\cdots c_{k+2}^{\ast}\big)^{\frac{1}{2}}\\
&=\varphi\big((c_1^{\ast}c_1) c_2 \cdots c_k (c_{k+1}c_{k+1}^{\ast})\cdots c_3^{\ast} c_{2}^{\ast}\big)^{\frac{1}{2}} \varphi\big((c_{k+2}^{\ast}c_{k+2})c_{k+3}\cdots (c_n c_n^{\ast})\cdots c_{k+3}^{\ast}\big)^{\frac{1}{2}},
\end{align*}
where we have used the trace property of $\varphi$.

Set $A^{2} = \varphi\big((c_1^{\ast}c_1) c_2 \cdots c_k (c_{k+1}c_{k+1}^{\ast})\cdots c_3^{\ast} c_{2}^{\ast}\big)$ and $B^{2}=\varphi\big((c_{k+2}^{\ast}c_{k+2})c_{k+3}\cdots (c_n c_n^{\ast})\cdots c_{k+3}^{\ast}\big)$.

For $A^{2}$, set:
\begin{itemize}
\item $\tilde{c}_1 = c_1^{\ast}c_1$,
\item for $j=2,\dots,k$, $\tilde{c}_j = c_j$, 
\item $\tilde{c}_{k+1}= c_{k+1}c_{k+1}^{\ast}$,
\item for $j=0,\dots,k-2$, $\tilde{c}_{k+2+j} = c_{k-j}^{\ast}$,
\end{itemize}
in such a way that $A^{2} = \varphi\big(\tilde{c}_1\tilde{c}_2\cdots \tilde{c}_{k}\cdots \tilde{c}_{2k}\big)$. Now, observe that $A^{2} = \varphi(a a^{\ast}) \geq 0$, with $a= c_1 \cdots c_{k+1}$, and therefore, by the induction hypothesis:
\begin{align*}
\bigg(\varphi &\big(\tilde{c}_1\tilde{c}_2\cdots \tilde{c}_{k}\cdots \tilde{c}_{2k}\big)\bigg)^{\frac{1}{2}}  \leq \prod_{l=1}^{2k}\prod_{s_{j} \in I_l(\bs{\tilde{c}})} \varphi\big((\tilde{c}_l \tilde{c}_l^{\ast})^{2^{s_{j}}}\big)^{2^{-(k+1)}} \\
&= \prod_{s_{j}\in I_1(\bs{\tilde{c}})}\varphi\big((\tilde{c}_{1} \tilde{c}_1^{\ast})^{2^{s_{j}}}\big)^{2^{-(k+1)}}\;  \prod_{l=2}^{k}  \prod_{s_{j}\in I_l(\bs{\tilde{c}})} \varphi\big((\tilde{c}_l \tilde{c}_l^{\ast})^{2^{s_{j}}}\big)^{2^{-(k+1)}} \\
&\cdot \prod_{s_{j}\in I_{k+1}(\bs{\tilde{c}})}\varphi\big((\tilde{c}_{k+1} \tilde{c}_{k+1}^{\ast})^{2^{s_{j}}}\big)^{2^{-(k+1)}} \;\prod_{l=k+2}^{2k}\prod_{s_{j}\in I_l(\bs{\tilde{c}})} \varphi\big((\tilde{c}_l \tilde{c}_l^{\ast})^{2^{s_{j}}}\big)^{2^{-(k+1)}} \;.
\end{align*}
Keeping in mind the definition of the $\tilde{c}_l$'s, we can write:
\begin{itemize}
\item $\varphi\big((\tilde{c}_{1} \tilde{c}_1^{\ast})^{2^{s_{j}}}\big) = \varphi\big((c_{1} c_1^{\ast})^{2^{s_{j}+1}}\big)$ for every $s_j \in I_1(\bs{\tilde{c}})$,
\item for $l=2,\dots,k$, $\varphi\big((\tilde{c}_{l} \tilde{c}_l^{\ast})^{2^{s_{j}}}\big) = \varphi\big((c_{l} c_l^{\ast})^{2^{s_{j}}}\big)$ for every $s_j \in I_{l}(\bs{\tilde{c}})$;
\item $\varphi\big((\tilde{c}_{k+1} \tilde{c}_{k+1}^{\ast})^{2^{s_{j}}}\big) = \varphi\big((c_{k+1} c_{k+1}^{\ast})^{2^{s_{j}+1}}\big)$ for every $s_j \in I_{k+1}(\bs{\tilde{c}})$,
\item for $l= k+2,\dots,2k$, $\varphi\big((\tilde{c}_{l} \tilde{c}_l^{\ast})^{2^{s_{j}}}\big) = \varphi\big((c_{2k-l+2} c_{2k-l+2}^{\ast})^{2^{s_{j}}}\big)$ for every $s_j \in I_l(\bs{\tilde{c}})$, so that:
$$ \prod_{l=k+2}^{2k}\prod_{s_{j}\in I_l(\bs{\tilde{c}})} \varphi\big((\tilde{c}_l \tilde{c}_l^{\ast})^{2^{s_{j}}}\big)^{2^{-(k+1)}} = \prod_{h=2}^{k}\prod_{s_{j}\in I_{2k-h+2}(\bs{\tilde{c}})} \varphi\big((c_h c_h^{\ast})^{2^{s_{j}}}\big)^{2^{-(k+1)}}.$$
\end{itemize}
In the end, if $\bs{c} = c_1\cdots c_n$, by setting:
\begin{itemize}
\item $I_1(\bs{c}) = I_1(\bs{\tilde{c}}) + 1 := \{ s_{j} +1: s_{j} \in I_1(\bs{\tilde{c}})\}$;
\item $I_{k+1}(\bs{c}) = I_{k+1}(\bs{\tilde{c}}) + 1 :=\{ s_{j} +1: s_{j} \in I_{k+1}(\bs{\tilde{c}})\}$;
\item for $l=2,\dots,k$, $I_{l}(\bs{c}) = I_{l}(\bs{\tilde{c}}) \cup I_{2k-l+2}(\bs{\tilde{c}})$,
\end{itemize}
in such a way that $\sum\limits_{s_{j} \in I_{l}(\bs{c})}2^{s_{j}} = 2^{k}$ for every $l=1,\dots,k+1$, one has:
$$ \bigg(\varphi\big(\tilde{c}_1\tilde{c}_2\cdots \tilde{c}_{k}\cdots \tilde{c}_{2k}\big)\bigg)^{\frac{1}{2}}
\leq \prod_{l=1}^{k+1}\prod_{s_{j} \in I_{l}(\bs{c})} \varphi\big( (c_l c_l^{\ast})^{2^{s_{j}}}\big)^{2^{-(k+1)}}.$$
 
In the same way, we obtain a similar estimate for $B = \varphi\big(d_1 d_2\cdots d_{2k}\big)^{\frac{1}{2}}$, by setting:
\begin{itemize}
\item $d_1= c_{k+2}^{\ast}c_{k+2}$;
\item $d_{j+1} = c_{k+j+2}$ for $j=1,\dots,k-1$;
\item $d_{k+1} = c_n c_{n}^{\ast}$;
\item $d_{k+j+1} = c_{n-j}^{\ast}$ for $j=1,\dots,k-1$.
\end{itemize}
More precisely, we obtain:
\begin{align*}
B &\leq \prod_{l=1}^{2k}\prod_{t_j \in I_l(\bs{d})} \varphi\big((d_l d_l^{\ast})2^{t_{j}}\big)^{2^{-(k+1)}} \\
&= \prod_{t_{j}\in I_1(\bs{d})}\varphi\big((d_1 d_1^{\ast})^{2^{t_{j}}}\big)^{2^{-(k+1)}}\;  \prod_{l=2}^{k}  \prod_{t_{j}\in I_l(\bs{d})} \varphi\big((d_l d_l^{\ast})^{2^{t_{j}}}\big)^{2^{-(k+1)}} \\
& \prod_{t_{j}\in I_{k+1}(\bs{d})}\varphi\big((d_{k+1} d_{k+1}^{\ast})^{2^{t_{j}}}\big)^{2^{-(k+1)}} \;\prod_{l=k+2}^{2k}\prod_{t_{j}\in I_l(\bs{d})} \varphi\big((d_l d_l^{\ast})^{2^{t_{j}}}\big)^{2^{-(k+1)}} 
\end{align*}

As for $A^2$,  by considering the definition of the $d_l$'s, we can write:
\begin{itemize}
\item $\varphi\big((d_{1} d_1^{\ast})^{2^{t_{j}}}\big) = \varphi\big((c_{k+2} c_{k+2}^{\ast})^{2^{t_{j}+1}}\big)$ for every $t_j \in I_1(\bs{d})$,
\item for $l=2,\dots,k$, $ \varphi\big((d_{l} d_l^{\ast})^{2^{t_{j}}}\big) = \varphi\big((c_{k+l+1} c_{k+l+1}^{\ast})^{2^{t_{j}}}\big)$ for every $t_j \in I_{l}(\bs{d})$, so that:
$$ \prod_{l=2}^{k}\prod_{t_{j}\in I_l(\bs{d})} \varphi\big((d_l d_l^{\ast})^{2^{t_{j}}}\big)^{2^{-(k+1)}} = \prod_{h=k+3}^{n-1}\prod_{t_{j}\in I_{h-k-1}} \varphi\big((c_h c_h^{\ast})^{2^{t_{j}}}\big)^{2^{-(k+1)}}.$$
\item $\varphi\big((d_{k+1} d_{k+1}^{\ast})^{2^{t_{j}}}\big) = \varphi\big((c_{n} c_{n}^{\ast})^{2^{t_{j}+1}}\big)$ for every $t_j \in I_{k+1}(\bs{d})$,
\item for $l= k+2,\dots,2k$, $\varphi\big((d_{l} d_l^{\ast})^{2^{t_{j}}}\big) = \varphi\big((c_{n-l+k+1} c_{n-l+k+1}^{\ast})^{2^{t_{j}+1}}\big)$ for every $t_j \in I_l(\bs{d})$, so that:
$$ \prod_{l=k+2}^{2k}\prod_{t_{j}\in I_l(\bs{d})} \varphi\big((d_l d_l^{\ast})^{2^{t_{j}}}\big)^{2^{-(k+1)}} = \prod_{h=k+3}^{n-1}\prod_{t_{j}\in I_{n-h+k+1}(\bs{d})} \varphi\big((c_h c_h^{\ast})^{2^{t_{j}}}\big)^{2^{-(k+1)}}.$$
\end{itemize}

In the end, if $\bs{c} = c_1\cdots c_n$, by setting:
\begin{itemize}
\item $I_{k+2}(\bs{c}) = I_1(\bs{d}) + 1 = \{ s_{j} +1: s_{j} \in I_1(\bs{\tilde{c}})\}$;
\item $I_{n}(\bs{c}) = I_{k+1}(\bs{d}) + 1$;
\item for $h=k+3,\dots,n-1$, $I_{h}(\bs{c}) = I_{h-k-1}(\bs{d}) \cup I_{n-h+k+1}(\bs{d})$,
\end{itemize}
in such a way that $\sum\limits_{t_{j} \in I_{l}(\bs{c})}2^{t_{j}} = 2^{k}$ for every $l=k+2,\dots,n$, one has:
$$ \bigg(\varphi\big(d_1 d_2\cdots d_{k}\cdots d_{2k}\big)\bigg)^{\frac{1}{2}}
\leq \prod_{l=k+2}^{n}\prod_{t_{j} \in I_{l}(\bs{c})} \varphi\big( (c_l c_l^{\ast})^{2^{t_{j}}}\big)^{2^{-(k+1)}}.$$

Henceforth, at the end we obtain:
\begin{align*}
|\varphi\big((c_1 \cdots c_{k+1}) (c_{k+2}\cdots c_{n})\big)| &\leq \prod_{l=1}^{n}\prod_{s_j \in I_l(\bs{c})} \varphi\big((c_l c_l^{\ast})^{2^{s_{j}}}\big)^{2^{-(k+1)}},
\end{align*}
with $\sum_{s_j \in I_{l}(\bs{c})}2^{s_{j}} = 2^{k}$ for every $l=1,\dots,n$.

Assume now that $n$ is odd. If $n=3$, apply the Cauchy-Schwarz inequality in the following way:
\begin{equation*} 
|\varphi\big(c_1 (c_2 c_3)\big)| \leq \varphi\big((c_1^{\ast}c_1)\big)^{\frac{1}{2}} \big(\varphi\big((c_2 c_2^{\ast})(c_3 c_3^{\ast})\big)\big)^{\frac{1}{2}} \leq \big(\varphi\big((c_1^{\ast}c_1)\big)\big)^{\frac{1}{2}} \big(\varphi\big((c_2^{\ast}c_2)^{2}\big)\big)^{\frac{1}{4}}\big(\varphi\big((c_3^{\ast}c_3)^{2}\big)\big)^{\frac{1}{4}}.
\end{equation*} 
Assume then that $k >1$ and that the result holds for all integers $n=2l+1$, with $l \leq k$. Let $n=2(k+1)+1 = 2k +3$ and apply the Cauchy-Schwarz inequality as follows:
\begin{align*}
|\varphi &\big((c_1 \cdots c_{\frac{n-1}{2}}) (c_{\frac{n+1}{2}}\cdots c_n)\big)| \leq \\
&\leq \bigg(\varphi\big((c_1^{\ast}c_1)c_2\cdots c_{\frac{n-3}{2}} (c_{\frac{n-1}{2}}c_{\frac{n-1}{2}}^{\ast})c_{\frac{n-3}{2}}^{\ast}\cdots c_{2}^{\ast}\big)\bigg)^{\frac{1}{2}}\; 
\bigg(\varphi\big( (c_{\frac{n+1}{2}}^{\ast}c_{\frac{n+1}{2}})c_{\frac{n+3}{2}}\cdots (c_n c_{n}^{\ast})c_{n-1}^{\ast}\cdots c_{\frac{n+3}{2}}^{\ast} \big)\bigg)^{\frac{1}{2}}.
\end{align*}
Now set $A^{2}=\varphi\big((c_1^{\ast}c_1)c_2\cdots c_{\frac{n-3}{2}} (c_{\frac{n-1}{2}}c_{\frac{n-1}{2}}^{\ast})c_{\frac{n-3}{2}}^{\ast}\cdots c_{2}^{\ast}\big)$ and $B^{2}=\varphi\big( (c_{\frac{n+1}{2}}^{\ast}c_{\frac{n+1}{2}})c_{\frac{n+3}{2}}\cdots (c_n c_{n}^{\ast})c_{n-1}^{\ast}\cdots c_{\frac{n+3}{2}}^{\ast}\big).$

As to $A^{2}$, set:
\begin{itemize}
\item $\tilde{c}_1 = c_1^{\ast}c_1$,
\item for $j=2,\dots,\frac{n-3}{2}, \tilde{c}_j = c_j$,
\item $\tilde{c}_{\frac{n-1}{2}} = c_{\frac{n-1}{2}}c_{\frac{n-1}{2}}^{\ast}$, 
\item for $j=1,\dots,\frac{n-5}{2}$, $\tilde{c}_{\frac{n-1}{2}+j} = c_{\frac{n-1}{2}-j}^{\ast}$,
\end{itemize}
in such a way that $ A^{2} = \varphi(\tilde{c}_1\cdots \tilde{c}_{2k}) = \varphi(a a^{\ast}) \geq 0$, and so, by the induction hypothesis for $2k= n-3$, we have:
\begin{align*}
\varphi &(\tilde{c}_1\cdots \tilde{c}_{2k})^{\frac{1}{2}} \leq \prod_{l=1}^{2k} \prod_{t_j \in I_l(\bs{\tilde{c}})} \varphi\big((\tilde{c}_l \tilde{c}_l^{\ast})^{2^{t_{j}}}\big)^{2^{-(k+1)}} \\
&= \prod_{t_j \in I_1(\bs{\tilde{c}})} \varphi\big((\tilde{c}_1 \tilde{c}_1^{\ast})2^{t_{j}}\big)^{2^{-(k+1)}}  \prod_{l=2}^{k} \prod_{t_j \in I_l(\bs{\tilde{c}})} \varphi\big((\tilde{c}_l \tilde{c}_l^{\ast})^{2^{t_{j}}}\big)^{2^{-(k+1)}} \\
& \prod_{t_j \in I_{k}(\bs{\tilde{c}})} \varphi\big((\tilde{c}_k \tilde{c}_k^{\ast})^{2^{t_{j}+1}}\big)^{2^{-(k+1)}} \cdot \prod_{l=k+1}^{2k}  \prod_{t_j \in I_l(\bs{\tilde{c}})} \varphi\big((\tilde{c}_l \tilde{c}_l^{\ast})^{2^{t_j}}\big)^{2^{-(k+1)}}
\end{align*}
with $\sum\limits_{t_j \in I_l(\bs{\tilde{c}})} 2^{t_{j}} = 2^{k-1}$ for every $l=1,\dots,2k$.
 
Again, by keeping in mind the definition of the $\tilde{c}_l$'s, we can write:
\begin{itemize}
\item $\varphi\big((\tilde{c}_{1} \tilde{c}_1^{\ast})^{2^{t_{j}}}\big) = \varphi\big((c_{1} c_1^{\ast})^{2^{t_{j}+1}}\big)$ for every $t_j \in I_1(\bs{\tilde{c}})$,
\item for $l=2,\dots, \frac{n-3}{2}=k$, $ \varphi\big((\tilde{c}_{l} \tilde{c}_l^{\ast})^{2^{t_{j}}}\big) = \varphi\big((c_{l} c_l^{\ast})^{2^{t_{j}}}\big)$ for every $t_j \in I_{l}(\bs{\tilde{c}})$;
\item $\varphi\big((\tilde{c}_{k+1} \tilde{c}_{k+1}^{\ast})^{2^{t_{j}}}\big) = \varphi\big((c_{\frac{n-1}{2}} c_{\frac{n-1}{2}}^{\ast})^{2^{t_{j}+1}}\big)$ for every $t_j \in I_{k+1}(\bs{\tilde{c}})$ \;(note that $k+ 1 = \dfrac{n-1}{2}$),
\item for $l= k+2,\dots,2k$, $\varphi\big((\tilde{c}_{l} \tilde{c}_l^{\ast})^{2^{s_{j}}}\big) = \varphi\big((c_{2k-l+2} c_{2k-l+2}^{\ast})^{2^{s_{j}+1}}\big)$ for every $s_j \in I_l(\bs{\tilde{c}})$, so that:
$$ \prod_{l=k+2}^{2k}\prod_{s_{j}\in I_l(\bs{\tilde{c}})} \varphi\big((\tilde{c}_l \tilde{c}_l^{\ast})^{2^{s_{j}}}\big)^{2^{-(k+1)}} = \prod_{h=2}^{\frac{n-3}{2}}\prod_{s_{j}\in I_{n-1-h}(\bs{\tilde{c}})} \varphi\big((c_h c_h^{\ast})^{2^{s_{j}}}\big)^{2^{-(k+1)}}.$$
\end{itemize}
In the end, if $\bs{c} = c_1\cdots c_n$, by setting:
\begin{itemize}
\item $I_1(\bs{c}) = I_1(\bs{\tilde{c}}) + 1 = \{ s_{j} +1: s_{j} \in I_1(\bs{\tilde{c}})\}$;
\item $I_{k+1}(\bs{c}) = I_{k+1}(\bs{\tilde{c}}) + 1$ ($k+1 = \frac{n-1}{2}$);
\item for $l=2,\dots,k = \frac{n-3}{2}$, $I_{l}(\bs{c}) = I_{l}(\bs{\tilde{c}}) \cup I_{n-1-l}(\bs{\tilde{c}})$,
\end{itemize}
so that $\sum\limits_{s_{j} \in I_{l}(\bs{c})}2^{s_{j}} = 2^{k}$ for every $l=1,\dots,k+1$, we have:
$$ \bigg(\varphi\big(\tilde{c}_1\tilde{c}_2\cdots \tilde{c}_{k}\cdots \tilde{c}_{2k}\big)\bigg)^{\frac{1}{2}}
\leq \prod_{l=1}^{\frac{n-1}{2}}\prod_{s_{j} \in I_{l}(\bs{c})} \varphi\big( (c_l c_l^{\ast})^{2^{s_{j}}}\big)^{2^{-(k+1)}},$$
with $\sum\limits_{t_j \in I_{l}(\bs{c})}2^{t_j} = 2^{k} = 2^{\frac{n-3}{2}}$ for every $l=1,\dots,\frac{n-1}{2}$.

Similarly, for $B^{2}$, set:
\begin{itemize}
\item $d_1:= c_{\frac{n+1}{2}}^{\ast}c_{\frac{n+1}{2}}$,
\item for $j=2,\dots,\frac{n-1}{2}$, $\tilde{d}_{j} = c_{\frac{n+1}{2}+j-1}$ (so $\tilde{d}_{\frac{n-1}{2}} = c_{n-1}$),
\item $d_{\frac{n+1}{2}} = c_n c_n^{\ast}$,
\item for all $j=1,\dots,\frac{n-3}{2}, d_{\frac{n+1}{2} +j} = c_{n-j}^{\ast}$,
\end{itemize}
such that $B^{2} = \varphi\big(d_1\cdots d_{n-1}\big)$, and we can apply the results for the string of even length $n-1= 2(k+1)$ to have:

\begin{align*}
&\varphi\big(d_1\cdots d_{n-1}\big)^{\frac{1}{2}} \leq \prod_{l=1}^{n-1}\prod_{t_j \in I_l(\bs{d})} \varphi\big((d_j d_j^{\ast})^{2^{t_{j}}}\big)^{2^{-(k+2)}} \\
&= \prod_{t_{j}\in I_1(\bs{d})}\varphi\big((d_1 d_1^{\ast})^{2^{t_{j}}}\big)^{2^{-(k+2)}}\;  \prod_{l=2}^{k+1}  \prod_{t_{j}\in I_l(\bs{d})} \varphi\big((d_l d_l^{\ast})^{2^{t_{j}}}\big)^{2^{-(k+2)}} \\
& \prod_{t_{j}\in I_{\frac{n+1}{2}}(\bs{d})}\varphi\big((d_{\frac{n+1}{2}} d_{\frac{n+1}{2}}^{\ast})^{2^{t_{j}}}\big)^{2^{-(k+2)}} \;\prod_{l=\frac{n+3}{2}}^{n-1}\prod_{t_{j}\in I_l(\bs{d})} \varphi\big((d_l d_l^{\ast})^{2^{t_{j}}}\big)^{2^{-(k+2)}} 
\end{align*}

As for $A^2$, by considering the definition of the $d_l$'s, we can write:
\begin{itemize}
\item $\varphi\big((d_{1} d_1^{\ast})^{2^{t_{j}}}\big) = \varphi\big((c_{k+2} c_{k+2}^{\ast})^{2^{t_{j}+1}}\big)$ for every $t_j \in I_1(\bs{d})$, being $\dfrac{n+1}{2} = k+2$;
\item for $l=2,\dots,k+1$, $\varphi\big((d_{l} d_l^{\ast})^{2^{t_{j}}}\big) = \varphi\big((c_{k+l+1} c_{k+l+1}^{\ast})^{2^{t_{j}}}\big)$ for every $t_j \in I_{l}(\bs{d})$, so that:
$$ \prod_{l=2}^{k+1}\prod_{t_{j}\in I_l(\bs{d})} \varphi\big((d_l d_l^{\ast})^{2^{t_{j}}}\big)^{2^{-(k+2)}} = \prod_{h=k+3}^{n-1}\prod_{t_{j}\in I_{h-k-1}(\bs{d})} \varphi\big((c_h c_h^{\ast})^{2^{t_{j}}}\big)^{2^{-(k+2)}}.$$
\item $\varphi\big((d_{k+2} d_{k+2}^{\ast})^{2^{t_{j}}}\big) = \varphi\big((c_{n} c_{n}^{\ast})^{2^{t_{j}+1}}\big)$ for every $t_j \in I_{k+2}(\bs{d})$ (being $k+2= \dfrac{n+1}{2}$);
\item for $l= k+3,\dots,n-1$, $\varphi\big((d_{l} d_l^{\ast})^{2^{t_{j}}}\big) = \varphi\big((c_{n-l+k+2} c_{n-l+k+2}^{\ast})^{2^{t_{j}+1}}\big)$ for every $t_j \in I_l(\bs{d})$, so that:
$$ \prod_{l=k+3}^{n-1}\prod_{t_{j}\in I_l(\bs{d})} \varphi\big((d_l d_l^{\ast})^{2^{t_{j}}}\big)^{2^{-(k+2)}} = \prod_{h=k+3}^{n-1}\prod_{t_{j}\in I_{n+k+2-h}(\bs{d})} \varphi\big((c_h c_h^{\ast})^{2^{t_{j}}}\big)^{2^{-(k+2)}}.$$
\end{itemize}

In the end, if $\bs{c} = c_1\cdots c_n$, by setting:
\begin{itemize}
\item $I_{k+2}(\bs{c}) = I_1(\bs{d}) + 1 = \{ s_{j} +1: s_{j} \in I_1(\bs{d})\}$;
\item $I_{n}(\bs{c}) = I_{k+1}(\bs{d}) + 1$;
\item for $h=k+3,\dots,n-1$, $I_{h}(\bs{c}) = I_{h-k-1}(\bs{d}) \cup I_{n+k+2-h}(\bs{d})$,
\end{itemize}
in such a way that $\sum\limits_{t_{j} \in I_{l}(\bs{c})}2^{t_{j}} = 2^{k+1} = 2^{\frac{n-1}{2}}$ for every $l=k+2,\dots,n$, we obtain:
$$ \bigg(\varphi\big(d_1 d_2\cdots d_{k}\cdots d_{2k}\big)\bigg)^{\frac{1}{2}}
\leq \prod_{l=k+2}^{n}\prod_{t_{j} \in I_{l}(\bs{c})} \varphi\big( (c_l c_l^{\ast})^{2^{t_{j}}}\big)^{2^{-(k+2)}},$$
yielding the desired conclusion.

\end{proof}

\subsection{Proof of Theorem \ref{Multiinvariance2}}

Without loss of generality, we can assume that $\mathrm{Inf}_i(f_N^{(h)}) \leq 1$ for every $i=1,\dots,N$, for every $h=1,\dots,n$.

The forthcoming proof is meant to generalize the proof of \cite[Theorem 1.3]{NourdinDeya} for Chebyshev sums. As such, it follows the same strategy.

\begin{proof}

Consider the auxiliary vectors $\bs{Z}^{(i)} = (\bs{Y}_1,\dots,\bs{Y}_{i-1},\bs{\mathcal{X}}_i,\dots,\bs{\mathcal{X}_N})$, with $\bs{Y}_i = \underbrace{(Y_i,\dots,Y_i)}_{d \text{ times }}$ and $\bs{\mathcal{X}}_i= (U_{h_1}(X_i),\dots,U_{h_d}(X_i))$ (we drop the dependence on $N$ to simplify the notation).
With these notation we can write: 
\begin{align}
\label{difference}
&\varphi\bigg(\prod_{s=1}^{k}\big(Q_N^{(1)}(\bs{\mathcal{X}}^{(N)})\big)^{m_{1,s}}\cdots \big(Q_N^{(n)}(\bs{\mathcal{X}}^{(N)})\big)^{m_{n,s}}\bigg) -\varphi\bigg(\prod_{s=1}^{k}\big(Q_N^{(1)}(\bs{Y})\big)^{m_{1,s}}\cdots \big(Q_N^{(n)}(\bs{Y})\big)^{m_{n,s}}\bigg) \\ 
&= \sum_{i=1}^{N}
\varphi\bigg(\prod_{s=1}^{k}\big(Q_N^{(1)}(\bs{Z}^{(i)})\big)^{m_{1,s}}\cdots \big(Q_N^{(n)}(\bs{Z}^{(i)})\big)^{m_{n,s}}\bigg) -\varphi\bigg(\prod_{s=1}^{k}\big(Q_N^{(1)}(\bs{Z}^{(i+1)})\big)^{m_{1,s}}\cdots \big(Q_N^{(n)}(\bs{Z}^{(i+1)})\big)^{m_{n,s}}\bigg)\nonumber  \\ 
&= \sum_{i=1}^{N}
\varphi\bigg(\prod_{s=1}^{k}\big(W_1^{(i)}+ V_1^{(i)}(\bs{\mathcal{X}}_i)\big)^{m_{1,s}}\cdots \big(W_n^{(i)}+ V_n^{(i)}(\bs{\mathcal{X}}_i)\big)^{m_{n,s}}\bigg) \nonumber\\ 
&-\varphi\bigg(\prod_{s=1}^{k}\big(W_1^{(i)}+ V_1^{(i)}(\bs{Y}_i))\big)^{m_{1,s}}\cdots \big(W_n^{(i)}+ V_n^{(i)}(\bs{Y}_i)\big)^{m_{n,s}}\bigg),\nonumber
\end{align}
where, for every $j=1,\dots,n$, we have set $Q_N^{(j)}(\bs{Z}^{(i)}) = W_j^{(i)} + V_j^{(i)}(\bs{\mathcal{X}}_i)$ \footnote{We omit again the dependence on $N$ to simplify the notation.}, with
\begin{equation}
\label{WN}
W_j^{(i)} = \sum_{i_1,\dots,i_d \in [N]\setminus \{i\}}f_N^{(j)}(i_1,\dots,i_d)Z_{i_1,1}^{(i)}\cdots Z_{i_d,d}^{(i)}
\end{equation}
(that is, $W_j^{(i)}$ is obtained by gathering together the summands where no $U_{h_p}(X_i)$ appears), and
\begin{equation}
\label{VN}
V_j^{(i)}(\bs{\mathcal{X}}_i) = \sum_{l=1}^{d}\sum_{\substack{i_1,\dots,i_{d-1}\in\\ [N]\setminus \{i\}}}f_N^{(j)}(i_1,\dots,i_{l-1},i,i_{l},\dots,i_{d-1})Z_{i_1,1}^{(i)}\cdots Z_{i_{l-1},l-1}^{(i)} U_{h_l}(X_i) Z_{i_{l}, l+1}^{(i)}\cdots Z_{i_{d-1},d}^{(i)}.
\end{equation}
Similarly, we set:
\begin{equation}
V_j^{(i)}(\bs{Y}_i) = \sum_{l=1}^{d}\sum_{\substack{i_1,\dots,i_{d-1}\\ \in [N]\setminus \{i\}}}f_N^{(j)}(i_1,\dots,i_{l-1},i,i_{l},\dots,i_{d-1})Z_{i_1,1}^{(i)}\cdots Z_{i_{l-1},l-1}^{(i)} Y_i Z_{i_{l}, l+1}^{(i)}\cdots Z_{i_{d-1},d}^{(i)}.
\end{equation}
Note that the polynomials $W_j^{(i)}$'s and $V_{l}^{(i)}$'s are self-adjoint operators. Recall again that 
$$Z_{i_j,j}^{(i)}= 
\begin{cases}
Y_{i_j} & \text{ if } i_j \leq i-1 \\
U_{h_j}(X_{i_j}) & \text{ if } i_j \geq i.
\end{cases}
$$
Thanks to the free binomial expansion (see the Lemma \ref{freebin}), applied simultaneously to each $W_j^{(i)} + V_j^{(i)}(\bs{\mathcal{X}}_i)$, we can write, for every $i=1,\dots,N$:
\begin{align*}
\varphi\bigg(\prod_{s=1}^{k}\big(W_1^{(i)}+ V_1^{(i)}(\bs{\mathcal{X}}_i)\big)^{m_{1,s}}&\cdots \big(W_n^{(i)}+ V_n^{(i)}(\bs{\mathcal{X}}_i)\big)^{m_{n,s}}\bigg) = \varphi\big(\prod_{s=1}^{k}(W_1^{(i)})^{m_1,s}\cdots (W_n^{(i)})^{m_{n,s}}\big)  \\
&+ \sum_{\bs{v} \in \mathcal{D}}\varphi\bigg(\prod_{s=1}^{k}\prod_{l=1}^{n}(W_{l}^{(i)})^{\alpha_{l,1}^{(s)}}V_l^{(i)}(\bs{\mathcal{X}}_i)^{\beta_{l,1}^{(s)}}\cdots (W_{l}^{(i)})^{\alpha_{l,r_l}^{(s)}}V_l^{(i)}(\bs{\mathcal{X}}_i)^{\beta_{l,r_l}^{(s)}}\bigg),
\end{align*}
with 
$$ \mathcal{D} = \{\bs{v} = (r_l^{(s)}, \bs{\alpha}_{l}^{(s)}, \bs{\beta}_{l}^{(s)}) \in \mathcal{D}_{n_{l,s},m_{l,s}}: s=1,\dots,k, \;l=1,\dots,n,\; n_{l,s}=1,\dots,m_{l,s} \},$$
$$ \mathcal{D}_{n_{l,s},m_{l,s}} = \{(r_l^{(s)}, \bs{\alpha}_{l}^{(s)}, \bs{\beta}_{l}^{(s)}): r_{l}^{(s)}=1,\dots,n_{l,s}, \sum_{h=1}^{r_l^{(s)}}\alpha_{l,h}^{(s)} = m_{l,s}- n_{l,s}, \sum_{h=1}^{r_l^{(s)}}\beta_{l,h}^{(s)} = n_{l,s}\}$$
and where at least one $\beta_{l,j}^{(s)} \geq 1$. Similarly we would have
\begin{align*}
\varphi\bigg(\prod_{s=1}^{k}\big(W_1^{(i)}+ V_1^{(i)}(\bs{Y}_i)\big)^{m_{1,s}}&\cdots \big(W_n^{(i)}+ V_n^{(i)}(\bs{Y}_i)\big)^{m_{n,s}}\bigg) = \varphi\big(\prod_{s=1}^{k}(W_1^{(i)})^{m_1,s}\cdots (W_n^{(i)})^{m_{n,s}}\big)  \\
&+ \sum_{\bs{v} \in \mathcal{D}}\varphi\bigg(\prod_{s=1}^{k}\prod_{p=1}^{n}(W_{l}^{(i)})^{\alpha_{l,1}^{(s)}}V_l^{(i)}(\bs{Y}_i)^{\beta_{l,1}^{(s)}}\cdots (W_{l}^{(i)})^{\alpha_{l,r_l}^{(s)}}V_l^{(i)}(\bs{Y}_i)^{\beta_{l,r_l}^{(s)}}\bigg),
\end{align*}
where at least on $\beta_{l,j}^{(s)} \geq 1$. Hence, in the difference (\ref{difference}), the term $\varphi\big(\prod\limits_{s=1}^{k}(W_1^{(i)})^{m_1,s}\cdots (W_n^{(i)})^{m_{n,s}}\big)$ cancels out.

Set 
$$ a_{s,l}^{(i)} = (W_{l}^{(i)})^{\alpha_{l,1}^{(s)}}V_l^{(i)}(A)^{\beta_{l,1}^{(s)}}\cdots (W_{l}^{(i)})^{\alpha_{l,r_l}^{(s)}}V_l^{(i)}(A)^{\beta_{l,r_l}^{(s)}},$$
for $s=1,\dots,k, l=1,\dots,n$ and $A \in \{\bs{\mathcal{X}}_i,\bs{Y}_i\}$.

For a fixed $i=1,\dots,N$, by virtue of the Lemma \ref{3.1bis} with $\mathcal{A}_j = \text{Alg}(1,U_{h_1}(X_j),\dots,U_{h_d}(X_j))$ for every $j > i$ and $\mathcal{A}_j = \text{Alg}(1,Y_j)$ for every $j < i$,
 $\mathcal{B} = \text{Alg}(1,U_{h_1}(X_i),\dots,U_{h_d}(X_i))$, and $\mathcal{D} = \text{Alg}(1,Y_i)$, if $\gamma:=\sum\limits_{s=1}^{k}\sum\limits_{l=1}^{n}\sum\limits_{p=1}^{r_l}\beta_{l,p}^{(s)} \leq 2$, for every $i=1,\dots,N$, the terms $\varphi\big(\prod_{s=1}^{k}\prod_{l=1}^{n}a_{s,l}^{(i)}\big)$ relative to $A=\bs{\mathcal{X}}_i$ either are zero or cancel with the corresponding ones associated with $A=\bs{Y}_i$.

Indeed, if $\gamma = 1$, in the argument of $x:= \varphi\big(\prod_{s=1}^k \prod_{p=1}^n a_{s,p}^{(i)}\big)$, we will have only a factor of the type $U_{h_l}(X_i)$, and so that $x=0$ by virtue of the first item in Lemma \ref{3.1bis}. If $\gamma = 2$, either we have only one power $\beta_{l,j}^{(s)} = 2$ or two different ones equal to $1$: in both cases, we will be in the situation where either the second or the third item in Lemma \ref{3.1bis} applies thanks to the hypothesis $h_i=h_{d-i+1}$ for $i=1,\dots, \lfloor \dfrac{d}{2}\rfloor$.

Therefore we can assume that $\gamma \geq 3$, and applying the triangle inequality, we are left to bound terms of the type
$$ |\varphi\big((a_{1,1}^{(i)}\cdots a_{1,n}^{(i)}) \cdots (a_{k,1}^{(i)}\cdots a_{k,n}^{(i)})\big)|,$$
where the corresponding parameter $\gamma=\sum\limits_{s=1}^{k}\sum\limits_{l=1}^{n}\sum\limits_{p=1}^{r_l}\beta_{l,p}^{(s)}$ verifies $\gamma \geq 3.$

The first step of our proof consists in applying the algorithm in Lemma \ref{algo}. If $kn$ is even (both if $k$ is even or $k$ is odd), we obtain straightforward from the algorithm that:
\begin{equation}
\label{pari}
|\varphi\big((a_{1,1}^{(i)}a_{1,2}^{(i)}\cdots a_{n,1}^{(i)}) \cdots (a_{k,1}^{(i)}\cdots a_{k,n}^{(i)})\big)|  \leq \prod_{s=1}^{k}\prod_{l=1}^{n}\prod_{t_j \in I_{l,s}(\bs{a})} \bigg( \varphi\big( (a_{s,l}^{(i)} (a_{s,l}^{(i)})^{\ast})^{2^{t_{j}}} \big)\bigg)^{2^{-\frac{kn}{2}}},
\end{equation}
with $\sum\limits_{t_j \in I_{l,s}(\bs{a})} 2^{t_j} = 2^{\frac{kn}{2}-1}$ for every $l=1,\dots,n$, while if $kn$ is odd:
\begin{align*}
|\varphi\big((a_{1,1}^{(i)}a_{1,2}^{(i)} &\cdots a_{1,n}^{(i)}) \cdots (a_{k,1}^{(i)}\cdots a_{k,n}^{(i)})\big)|  \\ 
&\leq \prod_{s=1}^{k} \prod_{l=1}^{\frac{n+1}{2}}\prod_{t_j \in I_{l,s}(\bs{a})} \bigg(\varphi\big( (a_{s,l}^{(i)} (a_{s,l}^{(i)})^{\ast})^{2^{t_{j}}}\big)\bigg)^{2^{-\frac{kn+1}{2}}} \prod_{l= \frac{n+3}{2}}^{n}\prod_{t_j \in I_{l,s}(\bs{a})} \bigg(\varphi\big( (a_{s,l}^{(i)} (a_{s,l}^{(i)})^{\ast})^{2^{t_{j}}}\big)\bigg)^{2^{-\frac{kn-1}{2}}} \numberthis \label{dispari}
\end{align*}
with $\sum\limits_{t_j \in I_{l,s}(\bs{a})}2^{t_{j}} = 2^{\frac{kn-1}{2}}$ for $l=1,\dots,\frac{n+1}{2}$, and $\sum_{t_j \in I_{l,s}(\bs{a})}2^{t_{j}} = 2^{\frac{kn-3}{2}}$, for $l=\frac{n+3}{2},\dots,n$, and for every $s=1,\dots,k$.

Looking at the definition of $a_{s,l}^{(i)}$, we note that in the product $a_{s,l}^{(i)}(a_{s,l}^{(i)})^{\ast}$ the factor $V_{l}^{(i)}(\bs{\mathcal{X}}_i)^{2\beta_{l,r_l}^{(s)}}$ appears exactly once, while for every $p=1,\dots,r_{l}-1$, $V_{l}^{(i)}(\bs{\mathcal{X}}_i)^{\beta_{l,p}^{(s)}}$ appears exactly twice.

Therefore, for every fixed $s=1,\dots,k$, $l=1,\dots,n$ and $t_j \in I_{l,s}(\bs{a})$, in the argument of $\varphi\big( \big(a_{s,l}^{(i)}(a_{s,l}^{(i)})^{\ast}\big)^{2^{t_{j}}}\big)$, (considering the traciality of $\varphi$) there are exactly $2^{t_{j}}(2r_l - 1)$ paired products of the type $(W_l^{(i)})^{t_1}(V_l^{(i)}(\bs{\mathcal{X}}_i))^{t_2}$, for certain integers $t_1, t_2$.

Moreover, as follows by the application of \cite[Proposition 3.5, Lemma 3.4]{NourdinDeya} (see Proposition \ref{prop02} and Lemma \ref{lemma01} respectively) to the random variables $W_l^{(i)}$ and $V_l^{(i)}(A)$, with $A \in \{\bs{\mathcal{X}}_i,\bs{Y}_i\}$, for every $r \geq 1$ there exist constants $C_{r,d}$ and $D_{r,d}$ such that:
$$ \varphi\big((W_j^{(i)})^{2r}\big) \leq C_{r,d} \;\mu_{2^{rd-1}}^{\hat{\bs{Z}}^{(i)}},$$
where $\hat{\bs{Z}}^{(i)} = (\bs{Y}_1,\dots,\bs{Y}_{i-1},\bs{\mathcal{X}}_{i+1},\dots,\bs{\mathcal{X}}_{n})$, and
$$\varphi\big( V_j^{(i)}(A)^{2r}\big) \leq D_{r,d}\;\mu_{2^{rd-1}}^{\bs{Z}^{(i)}} \big(\mathrm{Inf}_i(f_N^{(j)})\big)^{r}. $$

From here, the application of the generalized free H\"older inequality (Lemma \ref{lemma0} in the appendix) yields:
\begin{align*}
\varphi\big(&(a_{s,l}^{(i)} (a_{s,l}^{(i)})^{\ast})^{2^{t_{j}}}\big) \leq\\
& \mathrm{C}\;\bigg\{ \bigg[ \varphi\bigg(V_l^{2\beta_{l,r_l}^{(s)}2^{2^{t_{j}}(2r_j-1)}}\bigg) \bigg]^{2^{-2^{t_{j}}(2r_{l}-1)}} \bigg\}^{2^{t_{j}}} \cdot \prod_{p=1}^{r_l - 1} \bigg\{\bigg[ \varphi\bigg( V_l^{\beta_{l,p}^{(s)}2^{2^{t_{j}}(2r_l-1)}}\bigg)\bigg]^{2^{-2^{t_{j}}(2r_l-1)}} \bigg\}^{2^{t_{j}+1}} \\
&\leq \mathrm{C}\;\bigg[\bigg( \big(\mathrm{Inf}_i(f_N^{(l)}\big)^{\beta_{l,r_l}^{(s)}2^{2^{t_{j}}(2r_l-1)}}\bigg)^{2^{-2^{t_{j}}(2r_l-1)}} \bigg]^{2^{t_{j}}} \cdot \prod_{p=1}^{r_l-1} \bigg[ \bigg(\big(\mathrm{Inf}_i(f_N^{(l)})\big)^{\beta_{l,p}^{(s)}2^{2^{t_{j}}(2r_l-1)}}\bigg)^{2^{-2^{t_{j}}(2r_l-1)}} \bigg]^{2^{t_{j}}}\\
& \leq \mathrm{C}\;\prod_{p=1}^{r_l}\big(\mathrm{Inf}_i(f_N^{(l)})\big)^{\beta_{l,p}^{(s)}2^{t_{j}}} \;= \;\big(\mathrm{Inf}_i(f_N^{(l)})\big)^{2^{t_{j}}\sum_{p=1}^{r_l}\beta_{l,p}^{(s)}}  \numberthis \label{stima}
\end{align*}
(where the constant $C$ gathers all the estimates given by the application of the Proposition \ref{prop02} to the $W_j^{\alpha}$'s, since they do not depend on the influence function).

Therefore the product over all the $t_j$'s in $I_{l,s}(\bs{a})$, with $\sum_{j}2^{t_{j}} = 2^{\frac{kn}{2}-1}$ yields:
\begin{align*}
\prod_{t_j \in I_{l,s}(\bs{a})}\bigg(\varphi\big(&(a_{s,l}^{(i)} (a_{s,l}^{(i)})^{\ast})^{2^{t_{j}}}\big) \bigg)^{2^{-\frac{kn}{2}}} \;\leq \;  \big(\mathrm{Inf}_i(f_N^{(l)})\big)^{2^{-1}\sum_{p=1}^{r_l}\beta_{l,p}^{(s)}} \;\leq \; \big(\max_{h=1,\dots,n}\mathrm{Inf}_i(f_N^{(h)})\big)^{2^{-1}\sum_{p=1}^{r_l}\beta_{l,p}^{(s)}} 
\end{align*}
implying that:
\begin{align*}
\prod_{s=1}^{k}\prod_{l=1}^{n}\prod_{t_j \in I_{l,s}(\bs{a})}\bigg(\varphi\big((a_{s,l}^{(i)} (a_{s,l}^{(i)})^{\ast})^{2^{t_{j}}}\big) \bigg)^{2^{-\frac{kn}{2}}} &\leq \big(\max_{h=1,\dots,n}\mathrm{Inf}_i(f_N^{(h)})\big)^{2^{-1}\gamma} \\
&\leq \big(\max_{h=1,\dots,n}\mathrm{Inf}_i(f_N^{(h)})\big)^{\frac{3}{2}} \;=\; \max_{h=1,\dots,n}\big(\mathrm{Inf}_i(f_N^{(h)})\big)^{\frac{3}{2}},
\end{align*}
(keep in mind that we are assuming that $\mathrm{Inf}_i (f_N^{(h)}) \leq 1 $ for all $h$). 

Up to a combinatorial coefficient,  we have:
\begin{align*}
\sum_{i=1}^{N} |\varphi\big((a_{1,1}^{(i)}a_{1,2}^{(i)}\cdots a_{n,1}^{(i)}) \cdots &(a_{k,1}^{(i)}\cdots a_{k,n}^{(i)})\big)| \leq \sum_{i=1}^{N} \max_{h=1,\dots,n}\big(\mathrm{Inf}_i(f_N^{(h)})\big)^{\frac{3}{2}}\\
&\leq \sum_{i=1}^{N} \sum_{h=1}^{n}\big(\mathrm{Inf}_i(f_N^{(h)})\big)^{\frac{1}{2}}\big(\mathrm{Inf}_i(f_N^{(h)})\big) \\
&\leq \sum_{i=1}^{N}\sum_{h=1}^{n} (\tau_N^{(h)})^{\frac{1}{2}}\mathrm{Inf}_i(f_N^{(h)}) \\
&\leq n \max_{h=1,\dots,n}(\tau_N^{(h)})^{\frac{1}{2}} \sum_{i=1}^{N}\mathrm{Inf}_i(f_N^{(h)}) = n d \max_{h=1,\dots,n}(\tau_N^{(h)})^{\frac{1}{2}} ,\numberthis \label{chain}
\end{align*}
due to $\sum\limits_{i=1}^{N}\mathrm{Inf}_i(f_N^{(h)}) = d$, and the conclusion follows.

If $kn$ is odd, for every $s=1,\dots,k$, and $l=1,\dots,\frac{n-1}{2}$, the estimate in (\ref{stima}) gives
\begin{align*}
\prod_{t_{j} \in I_{l,s}(\bs{a})} \varphi\big( (a_{s,l}^{(i)} (a_{s,l}^{(i)})^{\ast})^{2^{t_{j}}} \big)^{2^{-\frac{kn-1}{2}}} &\leq  \bigg(\big(\mathrm{Inf}_i (f_N^{(l)})\big)^{2^{\frac{kn-3}{2}}\sum_{p=1}^{r_l}\beta_{l,p}^{(s)}}\bigg)^{2^{-\frac{kn-1}{2}}}\\
&= \big(\mathrm{Inf}_i (f_N^{(l)})\big)^{2^{-1}\sum_{p=1}^{r_l}\beta_{l,p}^{(s)}} \\
&\leq \bigg( \max_{h=1,\dots,n} \mathrm{Inf}_i(f_N^{(h)}) \bigg)^{2^{-1}\sum_{p=1}^{r_l}\beta_{l,p}^{(s)}},
\end{align*}
so that
$$ \prod_{l=1}^{\frac{n-1}{2}}\prod_{t_{j} \in I_{l,s}} \varphi\big( (a_{s,l}^{(i)} (a_{s,l}^{(i)})^{\ast})^{2^{t_{j}}} \big)^{2^{-\frac{kn-1}{2}}} \leq \bigg( \max_{h=1,\dots,n} \mathrm{Inf}_i(f_N^{(h)}) \bigg)^{2^{-1}\sum_{l=1}^{\frac{n-1}{2}}\sum_{p=1}^{r_l}\beta_{l,p}^{(s)}}.$$

Similarly, for $l=\frac{n+1}{2},\dots,n$, we will get to:
$$\prod_{l=\frac{n+1}{2}}^{n}\prod_{t_{j} \in I_{l,s}(\bs{a})} \varphi\big( (a_{s,l}^{(i)} (a_{s,l}^{(i)})^{\ast})^{2^{t_{j}}} \big)^{2^{-\frac{kn+1}{2}}} \leq \bigg( \max_{h=1,\dots,n} \mathrm{Inf}_i(f_N^{(h)}) \bigg)^{2^{-1}\sum_{l=\frac{n+1}{2}}^{n}\sum_{p=1}^{r_l}\beta_{l,p}^{(s)}},$$
(keep in mind that for $l=\frac{n+1}{2},\dots,n$, we have set $I_{l,s}(\bs{a}) = \{t_{j} \geq 1: \sum_{j}2^{t_{j}} = 2^{\frac{kn-1}{2}}\}$, while for $l=1,\dots, \frac{n-1}{2}$, $I_{l,s}(\bs{a}) = \{t_{j} \geq 1: \sum_{j}2^{t_{j}} = 2^{\frac{kn-3}{2}}\}$).
In the end, from (\ref{dispari}) we obtain:
\begin{align*}
|\varphi\big((a_{1,1}^{(i)}a_{1,2}^{(i)}\cdots a_{1,n}^{(i)}) &\cdots (a_{k,1}^{(i)}\cdots a_{k,n}^{(i)})\big)| \\ 
&\leq \prod_{s=1}^{k} \prod_{l=1}^{\frac{n-1}{2}}\prod_{t_{j} \in I_{l,s}(\bs{a})} \varphi\big( (a_{s,l}^{(i)} (a_{s,l}^{(i)})^{\ast})^{2^{t_{j}}} \big)^{2^{-\frac{kn-1}{2}}}\prod_{l=\frac{n+1}{2}}^{n}\prod_{t_{j} \in I_{l,s}(\bs{a})} \varphi\big( (a_{s,l}^{(i)} (a_{s,l}^{(i)})^{\ast})^{2^{t_{j}}} \big)^{2^{-\frac{kn+1}{2}}}
\\
&\leq \bigg( \max_{h=1,\dots,n} \mathrm{Inf}_i(f_N^{(h)}) \bigg)^{2^{-1}\gamma}\\
&\leq \bigg( \max_{h=1,\dots,n} \mathrm{Inf}_i(f_N^{(h)}) \bigg)^{\frac{3}{2}}.
\end{align*}

To conclude, it is sufficient to repeat the reasoning carried out in the chain of inequalities (\ref{chain}).

\end{proof}

\appendix
\section{Appendix}

\begin{teo}\cite[\textbf{Theorems 1.3,1.6}]{NourdinPeccatiSpeicher}
\label{FourthMoment}
Let $m\geq 2$ be an integer, and $\{g_N\}_N$ a sequence of mirror symmetric kernels in $L^{2}\big(\mathbb{R}_{+}^{m}\big)$ with $\|g_N \|_{L^{2}(\mathbb{R}_{+}^{m})} =1$. Consider the associated sequence of Wigner integrals $\{I_m^{S}(g_N)\}_N$, and denote by $\mathcal{S}\sim \mathcal{S}(0,1)$ a standard semicircular random variable. The following statements are equivalent:
\begin{enumerate}
\item $\{I_{m}^{S}(g_N)\}_{N} \stackrel{\text{ law }}{\longrightarrow} \mathcal{S}(0,1)$ as $N$ goes to infinity;
\item $ \lim\limits_{N \rightarrow \infty}\varphi[I_m^{S}(g_N)^4] = \varphi(\mathcal{S}^4)= 2$;
\item all the non trivial contractions of the kernels $g_N$ vanish in the limit, that is for every $r=1,\dots, m-1$:
$$\lim_{N \rightarrow \infty}\|g_N \stackrel{r}{\smallfrown}g_N \|_{L^{2}(\mathbb{R}_{+}^{2m-2r})}= 0.$$
\end{enumerate}
\end{teo}

\begin{teo}[\textit{\cite{NourdinPeccati}}]
\label{ConvWignerToFreePoisson}
Let $q$ be an even integer and let $Z(\lambda)$ denote a centered free Poisson random variable with parameter $\lambda >0$ on the fixed $W^{\ast}$-probability space $(\mathcal{A},\varphi)$. Let $\{I_{q}^{S}(g_n)\}_n$ be a sequence of Wigner integrals, with $g_n$ mirror symmetric kernel in $L^{2}(\mathbb{R}_{+}^{q})$, such that $\varphi\big(I_q^{S}(g_n)^2\big) = \| g_n\|_{L^{2}(\mathbb{R}_{+}^{q})}^2 = \varphi\big(Z(\lambda)^2 \big) = \lambda$. Then the following assertions are equivalent as $n$ goes to infinity:
\begin{itemize}
\item[(i)] $\{I_q^{S}(g_n)\}_{n}$ converges in distribution to $Z(\lambda)$;
\item[(ii)] $\varphi\big(I_q^{S}(g_n)^4\big) - 2\varphi\big(I_q^{S}(g_n)^3\big)$ converges to $\varphi\big(Z(\lambda)^4\big) - 2\varphi\big(Z(\lambda)^3\big) = 2\lambda^2 - \lambda$;
\item[(iii)] $\lim\limits_{n \rightarrow \infty}\|g_n \stackrel{\frac{q}{2}}{\smallfrown} g_n - g_n \|_{L^{2}(\mathbb{R}_{+}^{q})}= 0\;$, and $\lim\limits_{n\rightarrow \infty}\| g_n \stackrel{r}{\smallfrown}g_n\|= 0$ for every $r \in \{1,\dots, q-1\}\setminus\{\dfrac{q}{2}\}$.
\end{itemize}
\end{teo}

\begin{lemma}[\textit{\cite{NourdinDeya}}]
\label{freebin}
Let $A$ and $B$ be two random variables in $\mathcal{A}$. Then, for every positive integer $m$:
\begin{equation*}
(A + B)^{m} = A^{m} + \sum_{n=1}^{m}\sum_{(r, \mathbf{i}_{r},\mathbf{j}_{r})\in D_{m,n}}C_{m,n,r,\mathbf{i}_{r},\mathbf{j}_{r}} A^{i_1}B^{j_1}A^{i_2}B^{j_2}\cdots A^{i_r}B^{j_r},
\end{equation*}
where
$$ D_{m,n} = \{(r, \mathbf{i}_{r},\mathbf{j}_{r}) \in [m]\times \mathbb{N}^{r} \times \mathbb{N}^{r}: \sum_{l=1}^{r}i_l= m-n, \sum_{l=1}^{r}j_l = n\}.$$
\end{lemma}

\begin{lemma}\cite[\textbf{Lemma 12}]{Kargin}
\label{lemma0}
Let $X$ and $Y$ be two random variables belonging to a fixed free probability space $(\mathcal{A},\varphi)$. Then, for every $r\in \mathbb{N}$ and every choice of nonnegative integers $m_1,n_1,\dots, m_r,n_r$, the following H\"older type inequality holds:
$$|\varphi\big(X^{m_1}Y^{n_1}\cdots X^{m_r}Y^{n_r}\big)| \leq \big[\varphi\big(X^{2^{r}m_1}\big)\big]^{2^{-r}} \big[\varphi\big(Y^{2^{r}n_1}\big)\big]^{2^{-r}}\cdots \big[\varphi\big(X^{2^{r}m_r}\big)\big]^{2^{-r}}\big[\varphi\big(Y^{2^{r}n_r}\big)\big]^{2^{-r}}.$$
\end{lemma}

\begin{lemma}\cite[\textbf{Lemma 3.4}]{NourdinDeya}
\label{lemma01}
For every integer $r \geq 1$ and every sequence $X = \{X_i\}_i$ of random variables in $(\mathcal{A},\varphi)$, one has:
$$ |\varphi(X_{i_1}\cdots X_{i_r})| \leq \mu_{2^{r-1}}^{X},$$
where $\mu_k^{X} = \sup\limits_{\substack{1\leq l \leq k\\ i \geq 1}}\varphi(X_i^{2l})$ is the largest even moment of order $k$ of $X$.
\end{lemma}

\begin{prop}\cite[\textbf{Proposition 3.5}]{NourdinDeya}
\label{prop02}
Let $X_1,\dots, X_N$ be centered freely independent random variables and denote by $(\mu_k^{N})$ the corresponding sequence of the largest even moments, that is $\mu_k^{N} = \sup\limits_{\substack{i=1,\dots,N \\l=1,\dots,k}}\varphi(X_i^{2l})$. For $d\geq 1$, and for every integer $N$, consider a homogenous sum $P_N$ with unit-variance, mirror symmetric kernels $g_N:[N]^{d}\rightarrow \mathbb{R}$, vanishing on diagonals. Then, for every integer $r \geq 1$ there exists a constant $C_{r,d}$, only depending on $r$ and $d$, such that:
$$ \varphi\big(P_N(X_1,\dots,X_N)^{2r}\big) \leq C_{r,d}\;\mu_{2^{rd-1}}^{N}\bigg(\sum_{j_1,\dots,j_d=1}^{N}g_N(j_1,\dots,j_d)^{2}\bigg)^{r}.$$
\end{prop}

\textbf{Acknowledgements:}
The author wishes to express her sincere gratitude and appreciation to Professor Giovanni Peccati, for the invaluable support and for several helpful suggestions.

\end{document}